\documentclass[a4paper,10pt]{amsart}
\usepackage[english]{babel}
\usepackage{amsmath,tikz-cd}
\usepackage{amssymb}
\usepackage{mathrsfs}
\usepackage{enumerate}
\usepackage{mathtools,yfonts}
\usepackage{tikz,tkz-euclide}
\usepackage{pgfplots}
\usepackage{hyperref}
\usetikzlibrary{arrows}
\usepackage{makecell}

\usepackage{tcolorbox}
\usepackage[]{algorithm2e}

\newtheorem{thm}{Theorem}[section]
\newtheorem{Lemma}[thm]{Lemma}
\newtheorem{Proposition}[thm]{Proposition}
\newtheorem{Corollary}[thm]{Corollary}

\newtheorem*{thm*}{Theorem}

\theoremstyle{definition}

\newtheorem{Definition}[thm]{Definition}
\newtheorem{Remark}[thm]{Remark}

\newtheorem{Question}[thm]{Question}

\definecolor{wwwwww}{rgb}{0.4,0.4,0.4}

\DeclareMathOperator{\Cox}{Cox}
\DeclareMathOperator{\Pic}{Pic}

\DeclareMathOperator{\ch}{char}

\DeclareMathOperator{\mult}{mult}

\DeclareMathOperator{\Sing}{Sing}

\DeclareMathOperator{\rank}{rank}

\hypersetup{pdfpagemode=UseNone}
\hypersetup{pdfstartview=FitH}
		
\begin{document}

\title{Quartic and Quintic hypersurfaces with dense rational points}

\author[Alex Massarenti]{Alex Massarenti}
\address{\sc Alex Massarenti\\ Dipartimento di Matematica e Informatica, Universit\`a di Ferrara, Via Machiavelli 30, 44121 Ferrara, Italy}
\email{msslxa@unife.it}

\date{\today}

\begin{abstract}
Let $X_4\subset\mathbb{P}^{n+1}$ be a quartic hypersurface of dimension $n\geq 4$ over an infinite field $k$. We show that if either $X_4$ contains a linear subspace $\Lambda$ of dimension $h\geq \max\{2,\dim(\Lambda\cap\Sing(X_4))-2\}$ or has double points along a linear subspace of dimension $h\geq 3$, a smooth $k$-rational point and is otherwise general, then $X_4$ is unirational over $k$. This improves previous results by A. Predonzan and J. Harris, B. Mazur, R. Pandharipande for quartics. We also provide a density result for the $k$-rational points of quartic $3$-folds with a double plane over a number field, and several unirationality results for quintic hypersurfaces over a $C_r$ field. 
\end{abstract}

\maketitle
\setcounter{tocdepth}{1}
\tableofcontents

\section{Introduction}
An $n$-dimensional variety $X$ over a field $k$ is rational if it is birational to $\mathbb{P}^n_{k}$, while $X$ is unirational if there is a dominant rational map $\mathbb{P}^n_{k}\dasharrow X$. If $k$ is infinite and $X$ is unirational then the set $X(k)$ of the $k$-rational points of $X$ is Zariski dense in $X$.

Since the first half of the twentieth century the problem of establishing whether a degree $d$ hypersurface $X_d\subset\mathbb{P}^{n+1}$ is rational or unirational has been central in birational projective geometry \cite{Mo40}, \cite{Pr49}, \cite{Mor52}, \cite{IM71}, \cite{CG72}, \cite{Cil80}, \cite{Ko95}, \cite{Sh95}, \cite{HMP98}, \cite{HT00}, \cite{dF13}, \cite{BRS19}, \cite{RS19}. 

Quadric hypersurfaces with a smooth point are rational and as proven by J . Koll\'ar cubic hypersurfaces with a smooth point are unirational \cite{Kol02}. U. Morin proved that a general complex hypersurface $X_d\subset\mathbb{P}^{n+1}$ is unirational provided that $n$ is large enough with respect to $d$ \cite{Mo40}. This result has then been reproved, in a different way, by C. Ciliberto \cite{Cil80}, and extended to complete intersections by A. Predonzan \cite{Pr49}, K. Paranjape, V. Srinivas \cite{PS92}, and L. Ramero \cite{Ram90}.

Furthermore, J. Harris, B. Mazur and R. Pandharipande proved that $X_d\subset\mathbb{P}^{n+1}$ is unirational if the codimension of its singular locus is sufficiently big with respect to $n$ and $d$ \cite{HMP98}.

Before stating our main results on the unirationality of quartics we briefly survey the state of the art. By the work of U. Morin a general complex quartic $X_4\subset\mathbb{P}^{n+1}$ with $n\geq 5$ is unirational \cite{Mor36}, \cite{Mor52}. V. A. Iskovskikh and Y. I. Manin proved that the group of the birational automorphisms of a smooth quartic $X_4\subset\mathbb{P}^4$ is finite so that $X_4$ is not rational \cite{IM71}. 

Moreover, J. Harris and Y. Tschinkel showed that if $n\geq 3$ and $k$ is a number field then for some finite extension $k'$ of $k$ the set of $k'$-rational points of a smooth quartic $X_4\subset\mathbb{P}^{n+1}$ is dense in the Zariski topology; in other terms the $k$-rational points of $X_4$ are potentially dense \cite{HT00}. Despite this great amount of efforts the unirationality of the general quartic $X_4\subset\mathbb{P}^n$ for $n = 4,5$ is still an open problem and only special families of quartic $3$-folds, called quartics with with separable asymptotics, are known to be unirational \cite{Seg60}. For a nice survey on rationality and unirationality problems with a focus on their relation with the notion of rational connection we refer to A. Verra's paper \cite{Ver08}. We recall that a projective variety is rationally connected if any two of its points can be joined by a rational curve, and refer to C. Araujo's paper \cite{Ar05} for a survey on the subject.

In this paper we address the unirationality of quartics $X_4\subset\mathbb{P}^{n+1}$ containing a linear subspace whose dimension is larger than that of the singular locus of $X_4$, and containing a linear subspace with multiplicity two. Our main results in Theorems \ref{Quart_th} and \ref{Quart_th2} can be summarized as follows:
\begin{thm}\label{teorA}
Let $X_4\subset\mathbb{P}^{n+1}$ be a quartic hypersurface and $\Lambda\subset\mathbb{P}^{n+1}$ an $h$-plane. Assume that either
\begin{itemize}
\item[(i)] $n\geq 3, h\geq 2$, $\dim(\Lambda\cap\Sing(X_4))\leq h-2$, $X_4$ contains $\Lambda$ and is not a cone over a smaller dimensional quartic; or
\item[(ii)] $n\geq 4, h\geq 3$, $X_4$ has double points along $\Lambda$, a point $p\in X_4\setminus\Lambda$ and is otherwise general; 
\end{itemize}
then $X_4$ is unirational.  
\end{thm}

Furthermore, for quartic $3$-folds over a number field we have the following density result:

\begin{thm}\label{teorAbis}
Let $X_4\subset\mathbb{P}^{4}$ be a quartic hypersurface, over a number field $k$, having double points along a codimension two linear subspace $\Lambda\subset\mathbb{P}^{4}$, with a point $p\in X_4\setminus\Lambda$ and otherwise general. The set $X_4(k)$ of the $k$-rational points of $X_4$ is Zariski dense in $X_4$.   
\end{thm}

As Remark \ref{RP_Nec} shows the assumption on the existence of a point $p\in X_4\setminus\Lambda$ in the case of quartics singular along a linear subspace can not be dropped. Under extra assumptions on the base field or on the existence of rational points in special subloci of $X_4$ Theorem \ref{teorA} can be extended to smaller dimensional quartics. For instance, by Proposition \ref{sing1} a quartic surface $X_4\subset\mathbb{P}^3$ with double points along a line $\Lambda\subset\mathbb{P}^3$, a point $p\in X_4\setminus\Lambda$, a double point $q\in X_4\setminus\Lambda$ with $q\neq p$ and otherwise general is unirational. Furthermore, by the second part of Theorem \ref{Quart_th} a quartic $X_4\subset\mathbb{P}^{4}$ over a $C_r$ field, for which definition we refer to Remark \ref{lang}, with $r = 0,1$, having double points along a linear subspace $\Lambda$ with $\dim(\Lambda) = 1,2$, and otherwise general is unirational. Therefore, we are left with the following open questions:

\begin{Question}\label{questA}
Let $X_4\subset\mathbb{P}^{n+1}$ be a quartic hypersurface over a field $k$ such that either
\begin{itemize}
\item[(i)] $n = 3$, $X_4$ contains a line; or
\item[(ii)] $n = 3$, $X_4$ has double points along a linear subspace $\Lambda$ with $\dim(\Lambda) = 1,2$, a point $p\in X_4\setminus\Lambda$; or
\item[(iii)] $n = 2$, $X_4$ has double points along a line $\Lambda$, a point $p\in X_4\setminus\Lambda$;
\end{itemize}
and $X_4$ is otherwise general. Is then $X_4$ unirational?
\end{Question}

As we said $\text{(ii)}$ has a positive answer when the base field is $C_r$ with $r = 0,1$, while $\text{(i)}$ is open even over the complex numbers. Since any complex quartic $3$-fold contains a line $\text{(i)}$ actually asks about the unirationality of a general quartic $3$-fold and is probably one of the most interesting unirationality open problems. Since a quartic surface $X_4\subset\mathbb{P}^3$ with a double line is birational to a conic bundle $\text{(iii)}$ is interesting only when the base field is not algebraically closed.    

Note that by considering to the generic fiber of the resolution of the linear projection form $\Lambda$ as in the proof of Theorem \ref{Quart_th} a positive answer to Question \ref{questA} would extend the first part of Theorem \ref{teorA} to quartic hypersurfaces $X_4\subset\mathbb{P}^{n+1}$ with $n\geq 3$ containing a line, and the second part of Theorem \ref{teorA} to quartic hypersurfaces $X_4\subset\mathbb{P}^{n+1}$ with $n\geq 3$ having double points along either a line or a plane and an additional smooth point.    

\begin{Remark}\label{comp}
The main available results in the spirit of Theorem \ref{teorA} can be found in \cite{Pr49} and \cite{HMP98}. By \cite[Theorem 1]{Pr49} a quartic $X_4\subset\mathbb{P}^{n+1}$ containing an $h$-plane $\Lambda$ is unirational provided that $\Sing(X_4)\cap\Lambda = \emptyset$ and $h\geq 4$. The same result has been proved in \cite[Corollary 3.7]{HMP98} for $h\geq 97$. 

We would like to stress that both \cite{Pr49} and \cite{HMP98} as well as \cite{Ram90} provide unirationality results for hypersurfaces of arbitrary degree and general unirationality bounds when the base field is algebraically closed. 

In the case of quartics, Theorem \ref{teorA} (i) improves \cite[Theorem 1]{Pr49} and \cite[Corollary 3.7]{HMP98} in two directions: on one side, it is enough to have that $h\geq 2$, on the other side, $\Lambda$ is allowed to intersect the singular locus of $X_4$ as long as such intersection has codimension at least two in $\Lambda$. 
\end{Remark}

In the last section we investigate the unirationality of quintic hypersurfaces and divisors of bidegree $(3,2)$ in products of projective spaces. As a byproduct we get new examples of unirational but not stably rational varieties.

A variety $X$ is stably rational if $X\times \mathbb{P}^m$ is rational for some $m\geq 0$. Hence, a rational variety is stably rational, and a stably rational variety is unirational. The first examples of stably rational non-rational varieties had been given in \cite{BCSS85}, where the authors, using Ch\^atelot surfaces, constructed a complex non-rational conic bundle $T$ such that $T\times \mathbb{P}^3$ is rational.

In the last decade important advances on stable rationality have been made, especially for hypersurfaces in projective spaces \cite{Vo15}, \cite{CP16}, \cite{To16}, \cite{HKT16}, \cite{AO18}, \cite{Sc18}, \cite{BG18}, \cite{HPT18}, \cite{Sc19a}, \cite{Sc19b}, \cite{HPT19}. In \cite[Theorem 1.17]{CP16} J. L. Colliot-Th\'{e}l\`ene and A. Pirutka proved that a very general smooth complex quartic $3$-fold is not stably rational. In \cite[Corollary 1.4]{Sc19b} S. Schreieder gave the first examples of unirational non stably rational smooth hypersurfaces. A. Auel, C. B\"{o}hning and A. Pirutka proved that a very general divisor of bidegree $(3,2)$ in $\mathbb{P}^3\times\mathbb{P}^2$, over the complex numbers, is not stably rational. By Theorem \ref{333} we get that such a very general divisor is unirational but not stably rational. 

Furthermore, thanks to our unirationality results for divisors of bidegree $(3,2)$ we get new results on the unirationality of quintic hypersurfaces over $C_r$ fields and number fields. Finally, we would like to stress that since all the proofs presented in the paper are constructive it is possible, given the equation cutting out the hypersurface, to establish whether or not it is general in the required sense. 

\subsection*{Conventions on the base field, terminology and organization of the paper} 
All along the paper the base field $k$ will be of characteristic zero. Let $X$ be a variety over $k$. When we say that $X$ is rational or unirational, without specifying over which field, we will always mean that $X$ is rational or unirational over $k$. Similarly, we will say that $X$ has a point or contains a variety with certain properties meaning that $X$ has a $k$-rational point or contains a variety defined over $k$ with the required properties. 

In Section \ref{sec1} we will introduce the notation, prove some preliminary results about the relation between hypersurface in projective spaces and certain divisors in projective bundles, and give an immediate generalization of a unirationality criterion due to F. Enriques. In Section \ref{sec1} we will  investigate the unirationality of quartic hypersurfaces and cubic complexes that is complete intersections of a quadric and a cubic. Finally, in Section \ref{sec3} we will address the unirationality of quintics and divisors in products of projective spaces.

\section{Hypersurfaces and divisors in projective bundles}\label{sec1}

Let $a_0,\dots,a_{h+1} \in\mathbb{Z}_{\geq 0}$ be non negative integers, and consider the simplicial toric variety $\mathcal{T}_{a_0,\dots,a_{h+1}}$ with Cox ring
$$\Cox(\mathcal{T}_{a_0,\dots,a_{h+1}})\cong k[x_0,\dots,x_{n-h},y_0,\dots,y_{h+1}]$$
$\mathbb{Z}^2$-grading given, with respect to a fixed basis $(H_1,H_2)$ of $\Pic(\mathcal{T}_{a_0,\dots,a_{h+1}})$, by the following matrix
$$
\left(\begin{array}{cccccc}
x_0 & \dots & x_{n-h} & y_0 & \dots & y_{h+1}\\
\hline
1 & \dots & 1 & -a_0 & \dots & -a_{h+1}\\ 
0 & \dots & 0 & 1 & \dots & 1
\end{array}\right)
$$
and irrelevant ideal $(x_0,\dots,x_{n-h})\cap (y_0,\dots,y_{h+1})$.
Then
$$\mathcal{T}_{a_0,\dots,a_{h+1}}\cong \mathbb{P}(\mathcal{E}_{a_0,\dots,a_{h+1}})$$
with
$\mathcal{E}_{a_0,\dots,a_{h+1}}\cong\mathcal{O}_{\mathbb{P}^{n-h}}(a_0)\oplus\dots\oplus
\mathcal{O}_{\mathbb{P}^{n-h}}(a_{h+1})$. The secondary fan of $\mathcal{T}_{a_0,\dots,a_{h+1}}$ is as follows
$$
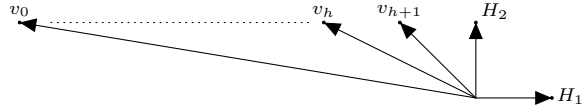
\begin{tikzpicture}[line cap=round,line join=round,>=triangle 45,x=1.0cm,y=1.0cm]
\clip(-7.,-0.2) rectangle (2.,1.3);
\draw [->,line width=0.4pt] (0.,0.) -- (-1.,1.);
\draw [->,line width=0.4pt] (0.,0.) -- (-2.,1.);
\draw [->,line width=0.4pt] (0.,0.) -- (-6.,1.);
\draw [->,line width=0.4pt] (0.,0.) -- (0.,1.);
\draw [->,line width=0.4pt] (0.,0.) -- (1.,0.);
\draw [line width=0.4pt,dotted] (-2.2,1.)-- (-5.6,1.);
\begin{scriptsize}
\draw [fill=black] (-1.,1.) circle (0.5pt);
\draw[color=black] (-1,1.15) node {$v_{h+1}$};
\draw [fill=black] (-2.,1.) circle (0.5pt);
\draw[color=black] (-2,1.15) node {$v_{h}$};
\draw [fill=black] (-6.,1.) circle (0.5pt);
\draw[color=black] (-6,1.15) node {$v_0$};
\draw [fill=black] (0.,1.) circle (0.5pt);
\draw[color=black] (0.25,1.15) node {$H_2$};
\draw [fill=black] (1.,0.) circle (0.5pt);
\draw[color=black] (1.25,0.02) node {$H_1$};
\end{scriptsize}
\end{tikzpicture}
$$
where $H_1 = (1,0)$ corresponds to the sections $x_0,\dots,x_{n-h}$, $H_2 = (0,1)$, and $v_i = (-a_i,1)$ corresponds to the section $y_i$ for $i = 0,\dots,h+1$.

\begin{Definition}\label{splitQB}
A divisor $D\subset\mathcal{T}_{a_0,\dots,a_{h+1}}$ of multidegree $(\delta_{d,0,\dots,0},\dots,\delta_{0,\dots,0,d}; d)$ is a hypersurface given by an equation of the following form
\stepcounter{thm}
\begin{equation}\label{Cox_g}
D := \left\lbrace\sum_{0\leq i_0\leq \dots \leq i_{h+1}\: |\: i_0+ \dots + i_{h+1}=d}\sigma_{i_0,\dots,i_{h+1}}(x_0,\dots,x_{n-h})y_0^{i_0}\dots y_{h+1}^{i_{h+1}} = 0\right\rbrace\subset \mathcal{T}_{a_0,\dots,a_{h+1}}
\end{equation}
where $\sigma_{i_0,\dots,i_{h+1}}\in k[x_0,\dots,x_{n-h}]_{\delta_{i_0,\dots,i_{h+1}}}$ and 
\stepcounter{thm}
\begin{equation}\label{compdeg}
\delta_{d,0,\dots,0}-da_0 = \delta_{d-1,1,0,\dots,0}-(d-1)a_0-a_1 = \dots = \delta_{0,\dots,0,d}-da_{h+1}.
\end{equation}
\end{Definition}
Without loss of generality we may assume that $a_0\geq a_1\geq\dots\geq a_{h+1}$ so that (\ref{compdeg}) yields $\delta_{d,0,\dots,0} \geq \delta_{0,d,0,\dots,0}\geq \dots\geq \delta_{0,\dots,0,d}$. 

\begin{Lemma}\label{hyp}
Let $X_d\subset\mathbb{P}^{n+1}$ be a hypersurface of degree $d$ having multiplicity $m$ along an $h$-plane $\Lambda$, and $\widetilde{X}_d$ the blow-up of $X_d$ along $\Lambda$ with exceptional divisor $\widetilde{E}\subset \widetilde{X}_d$. Then $\widetilde{X}_d$ is isomorphic to a divisor of multidegree 
$$(d,d-1,\dots,d-1,\dots,d-j,\dots,d-j,\dots,m,\dots,m; d-m)$$
in $\mathcal{T}_{1,0,\dots,0}$ where $d-j$ is repeated $\binom{h+j}{j}$ times for $j = 0,\dots,d-m$. The exceptional divisor $\widetilde{E}$ is a divisor of bidegree $(m,d-m)$ in $\mathbb{P}^{n-h}\times\mathbb{P}^{h}$.  
\end{Lemma}
\begin{proof}
We may assume that $\Lambda = \{z_0 = \dots = z_{n-h} = 0\}$ and 
$$X_d = \left\lbrace \sum_{m_0+\dots+m_{n-h} = m}z_0^{m_0}\dots z_{n-h}^{m_{n-h}}A_{m_0,\dots,m_{n-h}}(z_0,\dots,z_{n+1}) = 0\right\rbrace\subset\mathbb{P}^{n+1}$$
with $A_{m_0,\dots,m_{n-h}}\in k[z_0,\dots,z_{n+1}]_{d-m}$. The blow-up of $\mathbb{P}^{n+1}$ along $\Lambda$ is the simplicial toric variety $\mathcal{T}$ with Cox ring
$$\Cox(\mathcal{T})\cong k[x_0,\dots,x_{n-h},y_0,\dots,y_{h+1}]$$
$\mathbb{Z}^2$-grading given, with respect to a fixed basis $(H_1,H_2)$ of $\Pic(\mathcal{T})$, by the following matrix
$$
\left(\begin{array}{cccccccc}
x_0 & \dots & x_{n-h} & y_0 & y_1 & \dots & y_{h+1}\\
\hline
1 & \dots & 1  & 0 & 1 & \dots & 1\\ 
-1 & \dots & -1 & 1 & 0 & \dots & 0
\end{array}\right)
$$
and irrelevant ideal $(x_0,\dots,x_{n-h})\cap (y_0,\dots,y_{h+1})$. Substituting in the above matrix the first row with the sum of the rows, and then substituting in the matrix so obtained the first row with the difference of the first and the second row we get to the following grading matrix
\stepcounter{thm}
\begin{equation}\label{gM2}
\left(\begin{array}{cccccccc}
x_0 & \dots & x_{n-h} & y_0 & y_1 & \dots & y_{h+1}\\
\hline
1 & \dots & 1  & -1 & 0 & \dots & 0\\ 
0 & \dots & 0 & 1 & 1 & \dots & 1 
\end{array}\right)
\end{equation}
and hence $\mathcal{T} = \mathcal{T}_{1,0,\dots,0}$. The blow-down morphism is given by 
$$
\begin{array}{cccc}
 \phi: & \mathcal{T}_{1,0,\dots,0} & \longrightarrow & \mathbb{P}^{n+1}\\
  & (x_0,\dots,x_{n-h},y_0,\dots,y_{h+1}) & \mapsto & [x_0y_0:\dots :x_{n-h}y_0:y_1:\dots :y_{h+1}] 
\end{array}
$$
and the exceptional divisor is $E = \{y_0 = 0\}$. Hence, the strict transform of $X_d$ is defined by
\stepcounter{thm}
\begin{equation}\label{st}
\widetilde{X}_d = \left\lbrace \sum_{m_0+\dots+m_{n-h} = m}x_0^{m_0}\dots x_{n-h}^{m_{n-h}}A_{m_0,\dots,m_{n-h}}(x_0y_0,\dots,x_{n-h}y_0,y_1,\dots,y_{h+1}) = 0\right\rbrace\subset\mathcal{T}_{1,0,\dots,0}
\end{equation}
and the claim on the multidegree follows. Note that (\ref{gM2}) yields that $E \cong \mathbb{P}^{n-h}_{(x_0,\dots,x_{n-h})}\times\mathbb{P}^h_{(y_1,\dots,y_{h+1})}$ and hence $\widetilde{E} = \widetilde{X}_d\cap E \subset\mathbb{P}^{n-h}_{(x_0,\dots,x_{n-h})}\times\mathbb{P}^h_{(y_1,\dots,y_{h+1})}$ is a divisor of bidegree $(m,d-m)$.
\end{proof}

The following is a straightforward generalization of a unirationality criterion for conic bundles due to F. Enriques \cite[Proposition 10.1.1]{IP99}.  

\begin{Proposition}\label{Enr}
Let $f:X\rightarrow Y$ be a fibration over a unirational variety $Y$. Assume that there exists a unirational subvariety $Z\subset X$ such that $f_{|Z}:Z\rightarrow Y$ is dominant, consider the fiber product 
\[
  \begin{tikzpicture}[xscale=2.6,yscale=-1.3]
    \node (A0_0) at (0, 0) {$X_Z = X\times_{Y}Z$};
    \node (B) at (1, 0) {$X$};
    \node (A1_0) at (0, 1) {$Z$};
    \node (A1_1) at (1, 1) {$Y$};
    \path (A0_0) edge [->]node [auto] {$\scriptstyle{}$} (B);
    \path (B) edge [->]node [auto] {$\scriptstyle{f}$} (A1_1);
    \path (A1_0) edge [->]node [auto] {$\scriptstyle{f_{|Z}}$} (A1_1);
    \path (A0_0) edge [->]node [auto,swap] {$\scriptstyle{\widetilde{f}}$} (A1_0);
  \end{tikzpicture}
\]
and denote by $X_{Z,\eta}$ the generic fiber of $\widetilde{f}:X_Z\rightarrow Z$. Finally, assume that $X_{Z,\eta}$ is unirational over $k(Z)$ if and only if it has a $k(Z)$-rational point. Then $X$ is unirational. 
\end{Proposition}
\begin{proof}
Note that the dominant morphism $f_{|Z}:Z\rightarrow Y$ yields a rational section of   $\widetilde{f}:X_Z\rightarrow Z$. So $X_{Z,\eta}$ has a $k(Z)$-rational point and hence it is unirational over $k(Z)$. Therefore, $X_Z$ is unirational and hence $X$ is unirational as well.
\end{proof}

\begin{Remark}\label{QC}
A quadric hypersurface over a field $k$ with a smooth point is rational. Furthermore, by \cite[Theorem 1.2]{Kol02} a cubic hypersurface of dimension at least two with a point and which is not a cone is unirational. 
\end{Remark}

\begin{Remark}(Lang's theorem)\label{lang}
Fix a real number $r\in\mathbb{R}_{\geq 0}$. A field $k$ is $C_r$ if and only if  every homogeneous polynomial $f\in k[x_0,\dots,n_n]_d$ of degree $d > 0$ in $n+1$ variables with $n+1 > d^r$ has a non trivial zero in $k^{n+1}$.  

If $k$ is a $C_r$ field, $f_1,\dots, f_s\in k[x_0,\dots,n_n]_d$ are homogeneous polynomials of the same degree and $n+1 > sd^r$ then $f_1,\dots,f_s$ have a non trivial common zero in $k^{n+1}$ \cite[Proposition 1.2.6]{Poo17}. Furthermore, if $k$ is $C_r$ then $k(t)$ is $C_{r+1}$ \cite[Theorem 1.2.7]{Poo17}.  
\end{Remark}

In the last section we will need the following:

\begin{Proposition}\label{qb-hyp}
Let $D\subset \mathcal{T}_{a_0,\dots,a_{h+1}}\rightarrow \mathbb{P}^{n-h}$ be a divisor of multidegree $(\delta_{2,0,\dots,0},\dots,\delta_{0,\dots,0,2},2)$. Then $D$ is birational to a hypersurface $X_{\delta_{2,0,\dots,0}+2}^n\subset\mathbb{P}^{n+1}$ of degree $\delta_{2,0,\dots,0}+2$ having multiplicity $\delta_{2,0,\dots,0}$ along an $h$-plane $\Lambda$ and multiplicity two along an $(n-h-1)$-plane $\Lambda'$ such that $\Lambda\cap\Lambda' = \emptyset$. 
\end{Proposition}
\begin{proof}
Write $D\subset \mathcal{T}_{a_0,\dots,a_{h+1}}$ as in (\ref{Cox_g}) and dehomogenize with respect to $x_{n-h}$ and $y_{h+1}$ to get an affine hypersurface 
$$
X = \{f(x_0,\dots,x_{n-h-1},y_0,\dots,y_h) = 0\}\subset\mathbb{A}^{n+1}_{(x_0,\dots,x_{n-h-1},y_0,\dots,y_h)}.
$$
Now, we introduce a new variable that we will keep denoting by $x_{n-h}$ and homogenize $f$ in order to get a polynomial $\overline{f}(x_0,\dots,x_{n-h-1},y_0,\dots,y_h,x_{n-h})$ which is homogeneous of degree $\delta_{2,0,\dots,0}+2$. Note that $\overline{f}$ has the following form
\stepcounter{thm}
\begin{equation}\label{pol_proj}
\overline{f} = \sum_{i_0+\dots +i_{n-h} = \delta_{2,0,\dots,0}}x_0^{i_0}\dots x_{n-h}^{i_{n-h}}A_{i_0,\dots,i_{n-h}}(y_0,\dots,y_h,x_{n-h})
\end{equation}
with $A_{i_0,\dots,i_{n-h}}\in k[y_0,\dots,y_h,x_{n-h}]_2$ for all $0\leq i_0\leq\dots\leq i_{n-h}\leq \delta_{2,0,\dots,0}$. The hypersurface
$$
X_{\delta_{2,0,\dots,0}+2}^n = \{\overline{f}(x_0,\dots,x_{n-h-1},y_0,\dots,y_h,x_{n-h}) = 0\}\subset\mathbb{P}^{n+1}
$$
is birational to $X$ and hence it is birational to $D$ as well. To conclude set
$$
\Lambda = \{x_0 = \dots = x_{n-h} = 0\},\: \Lambda' = \{y_0 = \dots = y_{h} = x_{n-h} = 0\};
$$ 
and note that (\ref{pol_proj}) yields $\mult_{\Lambda}X_{\delta_{2,0,\dots,0}+2}^n = \delta_{2,0,\dots,0}$ and $\mult_{\Lambda'}X_{\delta_{2,0,\dots,0}+2}^n = 2$.
\end{proof}

\section{Cubic complexes and Quartics}\label{sec2}
In this section we investigate the unirationality of quartic hypersurface containing linear subspaces and quadrics, and of complete intersections of quadric and cubic hypersurfaces. The following is an immediate consequence of \cite[Theorem 2.1]{CMM07}.
\begin{Lemma}\label{Enr_gen}
Let $Y_{2,3} = Y_2 \cap Y_3\subset\mathbb{P}^{n+2}$ with $n\geq 3$ be a smooth complete intersection, of a quadric $Y_2$ and a cubic $Y_3$, defined over a field $k$ with $\ch(k)\neq 2,3$. If $Y_{2}$ contains a $2$-plane $\Pi$ and $Y_2$ is smooth at the general point of $\Pi$ then $Y_{2,3}$ is unirational. 
\end{Lemma}
\begin{proof}
For $n = 5$ the statement has been proven in \cite[Theorem 2.1]{CMM07}. Now, consider the incidence varieties  
$$
\begin{array}{ll}
W_2 & = \{(y,H) \: | \: y\in H\cap Y_{2}\}\subset Y_{2}\times \mathbb{G}(n-4,n-1); \\ 
W_3 & = \{(y,H) \: | \: y\in H\cap Y_{3}\}\subset Y_{3}\times \mathbb{G}(n-4,n-1); \\ 
W_{2,3} & = \{(y,H) \: | \: y\in H\cap Y_{2,3}\} = W_2\cap W_3\subset Y_{2,3}\times \mathbb{G}(n-4,n-1);
\end{array} 
$$
where $\mathbb{G}(n-4,n-1)$ is the Grassmannian of $5$-planes in $\mathbb{P}^{n+2}$ containing $\Pi$. Let $W_{2,\eta}, W_{3,\eta}, W_{2,3,\eta}$ be the generic fibers of the second projection onto $\mathbb{G}(n-4,n-1)$ from $W_2,W_3,W_{2,3}$ respectively. 

The $2$-plane $\Pi\subset Y_{2}$ yields a $2$-plane $\Pi'\subset W_{2,\eta}$ defined over $k(t_1,\dots,t_{3(n-3)})$, and since $Y_2$ is smooth at the general point of $\Pi$ we have that $W_{2,\eta}$ is smooth at the general point of $\Pi'$.

Hence, $W_{2,3,\eta}\subset \mathbb{P}^5_{k(t_1,\dots,t_{3(n-3)})}$ is a complete intersection over $k(t_1,\dots,t_{3(n-3)})$ of a quadric and a cubic satisfying the hypotheses of \cite[Theorem 2.1]{CMM07}, and so $W_{2,3,\eta}$ is unirational over the function field $k(t_1,\dots,t_{3(n-3)})$. Therefore, $W_{2,3}$ is unirational over the base field $k$, and to conclude it is enough to observe that the first projection $W_{2,3}\rightarrow Y_{2,3}$ is dominant.      
\end{proof}

\begin{Lemma}\label{hyp_ci}
Let $Q_0\subset\mathbb{P}^{n+1}$ be a smooth $(n-1)$-dimensional quadric and $X_d\subset\mathbb{P}^{n+1}$ an irreducible hypersurface of degree $d$ containing $Q_0$ with multiplicity one and otherwise general. Then there exists a complete intersection 
$$Y_{2,d-1} = Y_2\cap Y_{d-1}\subset\mathbb{P}^{n+2}$$ 
of a quadric $Y_2$ and a hypersurface $Y_{d-1}$ of degree $d-1$ such that:
\begin{itemize}
\item[(i)] $Y_{2,d-1}$ has a point $v\in Y_{2,d-1}$ of multiplicity $d-2$; and
\item[(ii)] the linear projection from $v$ yields a birational map $\pi_{v}: Y_{2,d-1}\dasharrow X_d$.
\end{itemize}
Furthermore, $Y_2$ is smooth at $v$ and $\mult_{v}Y_{d-1} = d-2$. 
\end{Lemma}
\begin{proof}
We may write $Q_0 = \{z_1 = Q = 0\}$ with $Q\in k[z_0,\dots,z_{n+1}]_2$. Then $X_d$ is of the form
$$X_d = \{z_1A+BQ = 0\}$$
with $A\in k[z_0,\dots,z_{n+1}]_{d-1}$ and $B\in k[z_0,\dots,z_{n+1}]_{d-2}$. Consider the quadric $Y_2 = \{z_1u-Q = 0\}$ with $z_0,\dots,z_{n+1},u$ homogeneous coordinates on $\mathbb{P}^{n+2}$. Denote by $CX_d$, $CQ_0$ the cones respectively over $X_d$ and $Q_0$ with vertex $v = [0:\dots :0:1]$. Then $CQ_0\subset Y_2$ and 
$$Y_2\cap CX_d = CQ_0 \cup \{z_1u-Q = A+uB = 0\}.$$
Set $Y_{d-1} = \{A+uB = 0\}$ and $Y_{2,d-1} = \{z_1u-Q = A+uB = 0\}$. Note that the tangent space of $Y_2$ at $v$ is the hyperplane $\{z_1 = 0\}$, the tangent cone of $Y_{d-1}$ at $v$ is given by $\{B = 0\}$, and the tangent cone of $Y_{2,d-1}$ at $v$ is cut out by $\{z_1 = B = 0\}$. Hence 
$$\mult_{v}Y_2 = 1,\: \mult_{v}Y_{d-1} = \mult_{v}Y_{2,d-1} = d-2.$$
Therefore, if $p\in Y_{2,d-1}$ is a general point the line $\left\langle v,p\right\rangle$ intersects $Y_{2,d-1}$ just in $v$ with multiplicity $d-2$ and in $p$ with multiplicity one. So the projection $\pi_{v}:Y_{2,d-1}\dasharrow \mathbb{P}^{n+1}$ is birational onto $\overline{\pi_{v}(Y_{2,d-1})}$ which must then be a hypersurface of degree $2(d-1)-(d-2) = d$. To conclude it is enough to note that since $Y_{2,d-1}$ is not a cone of vertex $v$ and $Y_{2,d-1}\subset CX_d$ we have $\overline{\pi_{v}(Y_{2,d-1})} = X_d$.  
\end{proof}

\begin{Proposition}\label{p_bun}
Let $X_d\subset\mathbb{P}^{n+1}$ be a hypersurface of degree $d$ having multiplicity $d-2$ along an $(n-1)$-plane $\Lambda\subset\mathbb{P}^{n+1}$. Assume that there is a quadric $Q_p$ in the quadric fibration induced by the projection from $\Lambda$ such that the quadric $Q_p\cap\Lambda$ is smooth and has a point and that $X_d$ is otherwise general.

Then there exist a rational surface $S$ and a variety $W$ with a morphism onto $S$ whose general fiber is a complete intersection of a quadric $W_2$ and a degree $d-3$ hypersurface $W_{d-3}$ in $\mathbb{P}^n$. Furthermore, geometrically $W_{d-3}$ is the union of $d-3$ linear components. Finally, if $W$ is unirational then there is a dominant rational map $W\dasharrow X_d$ and hence $X_d$ is unirational as well. 
\end{Proposition}
\begin{proof}
We may write $\Lambda = \{z_0 = z_1 = 0\}$ and
$$X_d = \{z_0^{d-2}A_1 +z_{0}^{d-3}z_1A_2+\dots +z_1^{d-2}A_{d-1} = 0\}$$
with $A_i \in k[z_0,\dots,z_{n+1}]_2$. Note that $X_d$ contains the smooth $(n-1)$-dimensional quadric $Q_p = \{z_1 = A_1 = 0\}$. Hence by Lemma \ref{hyp_ci} there is an irreducible complete intersection $Y_{2,d-1}\subset\mathbb{P}^{n+2}$ which is birational to $X_d$. 

Consider the quadric $Y_2 = \{z_1u-A_1 = 0\}\subset\mathbb{P}^{n+2}$ with homogeneous coordinates $z_0,\dots,z_{n+1},u$, and let $CX_d$ be the cone over $X_d$ as in the proof of Lemma \ref{hyp_ci}. The intersection $Y_2\cap CX_d$ has two components: the cone $CQ_p$ over $Q_p$ and the degree $d-1$ hypersurface
$$Y_{d-1} = \{z_0^{d-2}u+z_0^{d-3}A_2+z_0^{d-4}z_1A_3+\dots +z_1^{d-3}A_{d-1} = 0\}.$$
Set $v = [0:\dots :0:1]$ and $H = \{z_0 = z_1 = 0\}\subset\mathbb{P}^{n+2}$. Then
$$\mult_{v}Y_{d-1} = d-2,\: \mult_{H}Y_{d-1} = d-3.$$  
Take a general point $p\in H$. The lines through $p$ that intersect $Y_{d-1}$ at $p$ with multiplicity at least $d-2$ are parametrized by a hypersurface $W_{d-3}$ cut out in the $(n+1)$-dimensional projective space $\mathbb{P}(T_p\mathbb{P}^{n+2})$ of lines through $p$ by a polynomial in $k[z_0,z_1]_{d-3}$. 

Now, consider the cone $C\overline{Q}$ over the $(n-2)$-dimensional quadric $\overline{Q} = \{z_0 = z_1 = A_1 = 0\}$, and take a general point $p \in C\overline{Q}$. Note that $C\overline{Q}\subset Y_2$. The lines through $p$ that are contained in $Y_2$ are parametrized by a quadric hypersurface $W_2\subset\mathbb{P}(T_pY_2)$. 

Let $\mathcal{F} = \mathcal{T}_{Y_2|C\overline{Q}}$ be the restriction of the tangent sheaf of $Y_2$ to $C\overline{Q}$. Summing-up there is a subvariety $W_{2,d-3}\subset\mathbb{P}(\mathcal{F})$ with a surjective morphism $\rho:W_{2,d-3}\rightarrow C\overline{Q}$ whose fiber over a general point of $p\in C\overline{Q}$ is a complete intersection of a quadric and a hypersurface of degree $d-3$ in $\mathbb{P}^{n}$. Hence, $\dim(W_{2,d-3}) = 2n-3$.

By hypothesis $\overline{Q}$ has a point. Let $\overline{C}$ be a conic in $\overline{Q}$ through this point and $S$ the cone over $\overline{C}$ with vertex $v$. Then $S\subset C\overline{Q}$ is a rational surface. Set $W = \rho^{-1}(S)$ and $W_s = \rho^{-1}(s)$ for $s\in S$. A general point $w\in W$ represents a pair $(s,l_s)$ where $s\in S$ and $l_s$ is a line through $s$ which is contained in $Y_2$ and intersects $Y_{d-1}$ with multiplicity $d-2$ at $s$. Since $\deg(Y_{d-1}) = d-1$ the line $l_s$ intersects $Y_{d-1}$ just at one more point $x_{(s,l_s)}\in Y_2 \cap Y_{d-1} = Y_{2,d-1}$ and we get a rational map
$$
\begin{array}{cccc}
\psi: & W & \dasharrow & Y_{2,d-1}\\
 & (s,l_s) & \longmapsto & x_{(s,l_s)}.
\end{array}
$$
If $W$ is unirational in order to prove that $\psi$ is dominant it is enough to prove that the induced map $\overline{\psi}:\overline{W}\dasharrow\overline{Y}_{2,d-1}$ between the algebraic closures is dominant. Take a general point $p \in \overline{Y}_{2,d-1}$ and assume that $x_{(s,l_s)} = p$. Then $l_s$ lies in the tangent space of $Y_2$ at $p$ which is given by $\{L = 0\}$. Such tangent space intersects $\overline{S}$ in a conic, and further intersecting with $W_{d-3}$ we see that there are finitely may points $s\in \overline{S}$ such that $x_{(s,l_s)} = p$ for some $l_s\in W_s$. Furthermore, if $x_{(s,l_s)} = x_{(s,l_s')}$ then $l_s = l_s'$. Hence, $\overline{\psi}$ is generically finite and since $\dim(W) = \dim(Y_{2,d-1})$ we conclude that $\overline{\psi}$ is dominant. 

Finally, let $\pi_v:Y_{2,d-1}\dasharrow X_d$ be the dominant rational map in Lemma \ref{hyp_ci}. By considering the composition 
$$
\begin{tikzcd}
W \arrow[rr, "\psi", dashed] \arrow[rrd, "g", dashed] &  & {Y_{2,d-1}} \arrow[d, "\pi_v", dashed] \\
                                                       &  & X_{d}                                 
\end{tikzcd}
$$
we get a dominant rational map $g:W\dasharrow X_d$, and hence $X_d$ is unirational. 
\end{proof}

\begin{Proposition}\label{Y23s}
Let $Y_{2,3} = Y_2\cap Y_3\subset\mathbb{P}^{n+2}$ be a complete intersection of a quadric and a cubic of the following form
$$
Y_{2,3} =\{uz_1 - A = z_0^2u+z_0B+z_1C = 0\}
$$
with $A,B,C\in k[z_0,\dots,z_{n+1}]_2$ general. If the quadric $\overline{Q} = \{z_0 = z_1 = A = 0\}$ is smooth and has a point then $Y_{2,3}$ is unirational. 
\end{Proposition}
\begin{proof}
Take a point $p\in \overline{Q}$ and a general conic $\overline{C}\subset \overline{Q}$ through $p$. Let $S$ be the cone over $\overline{C}$ with vertex $v = [0:\dots:0:1]$. 

By Proposition \ref{p_bun} there exists a variety $W$ with a morphism onto $S$ whose general fiber is a complete intersection of a quadric and a hyperplane. Hence $W$ has a structure of quadric bundle $W\rightarrow S$ over $S$ with $(n-2)$-dimensional quadrics as fibers. Furthermore, the lines in $S$ yield a rational section of $W\rightarrow S$ and so $W$ is rational. 

Set $Y_2 = \{uz_1 - A = 0\}$ and $Y_3 = \{z_0^2u+z_0B+z_1C = 0\}$. Note that $S$ is contained in $Y$ and in the intersection of $Y_3$ with its tangent cone at $v$. By the proof of Proposition \ref{p_bun} a general point of $W$ corresponds to a pair $(s,l_s)$ with $s\in S$ and $l_s$ a line in $Y_2 = \{uz_1 - A = 0\}$ passing through $s\in S$ and intersecting $Y_3 = \{z_0^2u+z_0B+z_1C = 0\}$ with multiplicity two at $s\in S$. Associating to $(s,l_s)\in W$ the third point of intersection of $l_s$ with $Y_3$ as in the proof of Proposition \ref{p_bun} we get a dominant rational map $W\dasharrow Y_{2,3}$ and hence $Y_{2,3}$ is unirational. 
\end{proof}

\begin{Corollary}\label{nge3}
Let $X_4\subset\mathbb{P}^{n+1}$ be a quartic hypersurface having multiplicity two along an $(n-1)$-plane $\Lambda\subset\mathbb{P}^{n+1}$ with $n\geq 3$. Assume that there is a quadric $Q_p$ in the quadric fibration induced by the projection from $\Lambda$ such that the quadric $Q_{\Lambda} = Q_p\cap\Lambda$ is smooth and has a point and that $X_4$ is otherwise general. Then $X_4$ is unirational. 
\end{Corollary}
\begin{proof}
Up to a change of coordinates we may assume that $\Lambda = \{z_0 = z_1 = 0\}$, $X_4$ is given by
$$
X_4 = \{z_0^2A+z_0z_1B+z_1^2C = 0\}\subset\mathbb{P}^{n+1}
$$
and $Q_{\Lambda} = \{z_1 = A = 0\}$. By Lemma \ref{hyp_ci} there exist a complete intersection of a quadric and a cubic $Y_{2,3}$ and a dominant rational map $Y_{2,3}\dasharrow X_4$. To conclude it is enough to note that $Y_{2,3}$ is a complete intersection of the form covered by Proposition \ref{Y23s}.
\end{proof}

\begin{Corollary}
Let $X_4\subset\mathbb{P}^{n+1}$, with $n\geq 5$, be a quartic hypersurface containing an $(n-1)$-dimensional quadric $Q$ which contains a $2$-plane and otherwise general. Then $X_4$ is unirational. 
\end{Corollary}
\begin{proof}
By Lemma \ref{hyp_ci} $X_4$ is birational to a complete intersection $Y_{2,3} = Y_2\cap Y_3$. Furthermore, the proof of Lemma \ref{hyp_ci} shows that since $Q$ contains a $2$-plane $Y_2$ contains a $2$-plane as well. Hence, we conclude by Lemma \ref{Enr_gen}. 
\end{proof}

\begin{Corollary}\label{z}
Let $X_4\subset\mathbb{P}^{n+1}$, with $n\geq 2$, be a quartic hypersurface having multiplicity two along an $(n-1)$-dimensional quadric $Q$, with a point and otherwise general. Then $X_4$ is unirational. 
\end{Corollary}
\begin{proof}
Slightly modifying the proof of Lemma \ref{hyp_ci} we see that in this case $X_4$ is birational to a complete intersection $Y_2\cap Y_2'\subset\mathbb{P}^{n+2}$ of two quadrics and hence the claim follows from \cite[Proposition 2.3]{CSS87}. 
\end{proof}

In the following we will investigate the unirationality of quartic hypersurfaces by constructing explicit birational maps to divisors in products of projective spaces. In particular, we will get an improvement of Corollary \ref{nge3} when $n\geq 4$.

\begin{Lemma}\label{l0}
Consider the hypersurface 
$$
X_d = \{z_0^{d-2}(z_0L+Q) + z_0^{d-3}z_1 A_1+\dots + z_0z_1^{d-3}A_{d-3}+z_1^{d-2}A_{d-2} = 0\}\subset\mathbb{P}^{n+1}_{(z_0,\dots,z_{n+1})}
$$
where $A_i = A_i^2 + z_0A_i^1+z_0^2A_i^0$, and $L,A_i^1\in k[z_1,\dots,z_{n+1}]_1$, $Q,A_i^2\in k[z_1,\dots,z_{n+1}]_2$, $A_i^0\in k$ for $i = 1,\dots,d-2$. Then $X_d$ is birational to the divisor
$$
Y_{(d-1,2)} = \left\{\sum_{i=0}^{d-1}x_0^{d-1-i}x_1^iB_i = 0\right\}\subset\mathbb{P}^1_{(x_0,x_1)}\times \mathbb{P}^{n}_{(w_1,\dots, w_{n+1})}
$$
of bidegree $(d-1,2)$, where $B_0 = w_1L+A_1^0w_1^2$, $B_1 = Q+w_1A_1^1+w_1^2A_2^0$, $B_j = A_{j-1}^2+w_1A_j^1+w_1^2A_{j+1}^0$ for $j = 2,\dots,d-3$, $B_{d-2} = A_{d-3}^2+w_1A_{d-2}^1$, $B_{d-1} = A_{d-2}^2$.
\end{Lemma}
\begin{proof}
Note that $X_d$ passes through the point $p = [1:0:\dots:0]$, and the rational map 
$$
\begin{array}{llll}
\varphi: & \mathbb{P}^{n+1}_{(z_0,\dots,z_{n+1})} & \dasharrow & \mathbb{P}^{n+1}_{(w_0,\dots,w_{n+1})}\\
 & [z_0:\dots:z_{n+1}] & \mapsto & [z_0L:z_1^2:z_1z_2:\dots :z_1z_{n+1}]
\end{array}
$$
is birational with birational inverse 
$$
\begin{array}{llll}
\varphi^{-1}: & \mathbb{P}^{n+1}_{(w_0,\dots,w_{n+1})} & \dasharrow & \mathbb{P}^{n+1}_{(z_0,\dots,z_{n+1})}\\
 & [w_0:\dots:w_{n+1}] & \mapsto & [w_0w_1:w_1L:w_2L:\dots :w_{n+1}L]
\end{array}
$$
where $L = L(w_1,\dots,w_{n+1})$. Note that $\varphi^{-1}$ contracts the divisor $\{L = 0\}$ to the point $p$. The strict transform of $X_d$ via $\varphi^{-1}$ is given by 
\begin{small}
$$
\widetilde{X}_d = \{w_0^{d-2}(w_0w_1L+LQ) + w_0^{d-3}(w_0^2w_1^2A_1^0 + w_0w_1LA_1^1+L^2A_1^2)+ \dots +L^{d-3}(w_0^2w_1^2A_{d-2}^0 + w_0w_1LA_{d-2}^1+L^2A_{d-2}^2)=0\}.
$$
\end{small}
which we rewrite as 
\begin{small}
$$
\widetilde{X}_d = \{w_0^{d-1}(w_1L+w_1^2A_1^0)+w_0^{d-2}L(Q+w_1A_1^1+w_1^2A_2^0)+\dots +w_0L^{d-2}(A_{d-3}^2+w_1A_{d-2}^1) + L^{d-1}A_{d-2}^2 = 0\}.
$$
\end{small}
Finally, substituting $w_0 = \frac{x_0}{x_1}L$ we get the equation cutting out the divisor $Y_{(d-1,2)}\subset \mathbb{P}^1_{(x_0,x_1)}\times \mathbb{P}^{n}_{(w_1:\dots: w_{n+1})}$ in the statement.
\end{proof}

\begin{Proposition}\label{p0}
Let $Y_{(3,2)}$ be a general divisor of the form
$$
Y_{(3,2)} = \left\{\sum_{i=0}^{3}x_0^{3-i}x_1^iB_i = 0\right\}\subset\mathbb{P}^1_{(x_0,x_1)}\times \mathbb{P}^{n}_{(w_1,\dots, w_{n+1})}
$$
where $B_0 = w_1L+A_1^0w_1^2$, $B_1 = Q+w_1A_1^1+w_1^2A_2^0$, $B_{2} = A_{1}^2+w_1A_2^1$, $B_{3} = A_2^2$, and $L,A_i^1\in k[w_1,\dots,w_{n+1}]_1$, $Q,A_i^2\in k[w_1,\dots,w_{n+1}]_2$, $A_i^0\in k$ for $i = 1,2$ are general. If $n\geq 4$ then $Y_{(3,2)}$ is unirational.
\end{Proposition}
\begin{proof}
Consider the rational map
\stepcounter{thm}
\begin{equation}\label{eta}
\begin{array}{llll}
\eta: & \mathbb{P}^1_{(x_0,x_1)}\times \mathbb{P}^{n}_{(w_1,\dots, w_{n+1})} & \dasharrow & \mathbb{P}^{2}_{(u_1,u_2,u_{3})}\\
 & ([x_0:x_1],[w_1:\dots: w_{n+1}]) & \mapsto & [\eta_1:\eta_2:\eta_{3}]
\end{array}
\end{equation}
where $\eta_i = B_ix_0^{2}+B_{i+1}x_0x_1+\dots+B_{3}x_0^{i-1}x_1^{3-i}$ for $i = 1,2,3$. By \cite[Theorem 1.1 (ii)]{Ott15} the rational map
$$
\begin{array}{lll}
\mathbb{P}^1_{(x_0,x_1)}\times \mathbb{P}^{n}_{(w_1,\dots, w_{n+1})} & \dasharrow & \mathbb{P}^{2}_{(u_1,u_2,u_{3})}\times \mathbb{P}^{n}_{(w_1,\dots, w_{n+1})}\\
([x_0:x_1],[w_1:\dots: w_{n+1}]) & \mapsto & (\eta([x_0:x_1],[w_1:\dots: w_{n+1}]),[w_1:\dots: w_{n+1}])
\end{array}
$$
yields a small transformation $\eta^{+}:Y_{(3,2)}\dasharrow Y_{(3,2)}^{+}$, where $Y_{(3,2)}^{+}\subset \mathbb{P}^{2}_{(u_1,u_2,u_{3})}\times \mathbb{P}^{n}_{(w_1,\dots, w_{n+1})}$ is cut out by the minors of order three of the following matrix
\stepcounter{thm}
\begin{equation}\label{Ott_Mx}
M_{(u_1,u_2,u_{3})} = 
\left(
\begin{array}{ccc}
0 & u_1 & B_0\\ 
-u_1 & u_2 & B_1\\ 
-u_{2} & u_{3} & B_{2}\\ 
-u_{3} & 0 & B_{3}
\end{array} 
\right).
\end{equation}
Consider the point $p = ([1:0],[0:\dots:0:1])\in Y_{(3,2)}$, its image 
$$q = ([B_1(0,\dots,0,1),B_{2}(0,\dots,0,1)],B_{3}(0,\dots,0,1)]) = ([Q(0,\dots,0,1),A_{1}^2(0,\dots,0,1),A_{2}^2(0,\dots,0,1)])$$
via $\eta$ and set $\overline{u}_1 = Q(0,\dots,0,1),\overline{u}_{2} = A_{1}^2(0,\dots,0,1),\overline{u}_{3} = A_{2}^2(0,\dots,0,1)$. Let $F_{\overline{u}}$ be the fiber of the first projection $\pi_1:\mathbb{P}^{2}_{(u_1,u_2,u_{3})}\times \mathbb{P}^{n}_{(w_1,\dots, w_{n+1})}\rightarrow \mathbb{P}^{2}_{(u_1,u_2,u_{3})}$ over $\overline{u} = [\overline{u}_1:\overline{u}_2:\overline{u}_{3}]$. Then 
$$
F_{\overline{u}} = \{\rank(M_{(\overline{u}_1,\overline{u}_2,\overline{u}_{3})}) < 3\}\subset \mathbb{P}^{n}_{(w_1,\dots, w_{n+1})}
$$
is a complete intersection of two quadrics. Note that $q\in F_{\overline{u}}$ and since the $\overline{u}_i$ are general $F_{\overline{u}}$ is smooth. Therefore, if $n\geq 4$ \cite[Proposition 2.3]{CSS87} yields the unirationality of $F_{\overline{u}}$. The strict transform of $F_{\overline{u}}$ via $\eta^{+}$ is given by
$$
\widetilde{F}_{\overline{u}} = \left\lbrace 
\rank\left(\begin{array}{ccc}
\overline{u}_1 & \overline{u}_2 & \overline{u}_3 \\ 
B_1x_0^2+B_2x_0x_1+B_3x_1^2 & B_2x_0^2+B_3x_0x_1 & B_3x_0^2
\end{array}\right) < 2\right\rbrace \subset \mathbb{P}^1_{(x_0,x_1)}\times \mathbb{P}^{n}_{(w_1,\dots, w_{n+1})}.
$$
So $\widetilde{F}_{\overline{u}}$ is unirational and maps dominantly onto $\mathbb{P}^1_{(x_0,x_1)}$. Finally, to conclude it is enough to note that $Y_{(3,2)}\rightarrow \mathbb{P}^1_{(x_0,x_1)}$ is a fibration in quadric hypersurfaces and to apply Proposition \ref{Enr} and Remark \ref{QC}.
\end{proof}

\begin{Proposition}\label{Quart_2}
Let $X_4\subset\mathbb{P}^{n+1}$ be a quartic hypersurface having double points along a codimension two linear subspace $\Lambda\subset\mathbb{P}^{n+1}$, with a point $p\in X_4\setminus\Lambda$ and otherwise general. If $n\geq 4$ then $X_4$ is unirational.
\end{Proposition}
\begin{proof}
The equation of $X_4\subset\mathbb{P}^{n+1}$ can be written as in Lemma \ref{l0} for $d = 4$. Hence the claim follows from Lemma \ref{l0} and Proposition \ref{p0}.
\end{proof}

For quartic $3$-folds we have the following density result:

\begin{Proposition}\label{Quart_3_Dens}
Let $X_4\subset\mathbb{P}^{4}$ be a quartic hypersurface, over a number field $k$, having double points along a codimension two linear subspace $\Lambda\subset\mathbb{P}^{4}$, with a point $p\in X_4\setminus\Lambda$ and otherwise general. The set $X_4(k)$ of the $k$-rational points of $X_4$ is Zariski dense in $X_4$.
\end{Proposition}
\begin{proof}
By Lemma \ref{l0} it is enough to prove that a general divisor $Y_{3,2}\subset\mathbb{P}^1\times\mathbb{P}^3$ as in Proposition \ref{p0} has dense $k$-points. Consider the $2$-plane $H = \{x_1 = w_1 = 0\}\subset Y_{3,2}$ and the rational map
$$
\begin{array}{lcll}
\eta': & Y_{(3,2)} & \dasharrow & \mathbb{P}^{2}_{(u_1,u_2,u_{3})}\\
 & ([x_0:x_1],[w_1:\dots: w_{4}]) & \mapsto & [\eta_1:\eta_2:\eta_{3}]
\end{array}
$$
induced by (\ref{eta}). Note the $\eta'$ maps $H$ dominantly onto $\mathbb{P}^{2}_{(u_1,u_2,u_{3})}$. Take a general point $p\in H$, its image $q = \eta'(p)$ and a general line $L = \{\alpha_1u_1+\alpha_2u_2+\alpha_3u_3 = 0\}\subset \mathbb{P}^{2}_{(u_1,\dots,u_{3})}$ through $q$. Set $S_{L,p} = \eta'^{-1}(L)$. Since $Y_{(3,2)}$ is general $\eta'$ restricts to a morphism $\eta'_{|S}:S\rightarrow\mathbb{P}^1$. A straightforward computation shows that $S_{L,p}$ is smooth and by the argument in the second part of the proof of Proposition \ref{p0} the general fiber of $\eta'_{|S_{L,p}}:S_{L,p}\rightarrow\mathbb{P}^1$ is a smooth curve of genus one. 

Furthermore, $C = H\cap S_{L,p} = \{\alpha_1B_1+\alpha_2B_2+\alpha_3B_3 = x_1 = w_1 = 0\}$. Hence, since $Y_{3,2}$ is general, $C$ is a smooth conic with a point $p\in C$. Note that a general fiber of $\eta'_{|S_{L,p}}:S_{L,p}\rightarrow\mathbb{P}^1$ intersects a general fiber of the quadric bundle $Y_{3,2}\rightarrow\mathbb{P}^1$ in a $0$-dimensional scheme of degree eight. So, $C\subset S_{L,p}$ is a rational $4$-section of $\eta'_{|S_{L,p}}$. 

Consider the fiber product $T = C\times_{L}S$. Since $p,L$ and $Y_{3,2}$ are general the ramification divisor of $\eta'_{|C}:C\rightarrow\mathbb{P}^1$ is disjoint from the singular fibers of $\eta'_{|S_{L,p}}$. Hence, $T$ is smooth and the proof of \cite[Theorem 8.1]{HT00} goes through. So the set $S_{L,p}(k)$ of the $k$-rational points of $S_{L,p}$ is Zariski dense in $S_{L,p}$. Finally, letting the point $p$ vary in $H$ and the line $L$ vary among the lines passing through $q = \eta'(p)$ we get the claim.  
\end{proof}

The following is our main result on the unirationality of quadric hypersurfaces having double points along a linear subspace.

\begin{thm}\label{Quart_th}
Let $X_4\subset\mathbb{P}^{n+1}$ be a quartic hypersurface having double points along an $h$-plane $\Lambda\subset\mathbb{P}^{n+1}$, with a point $p\in X_4\setminus\Lambda$ and otherwise general. If $n\geq 4$ and $h\geq 3$ then $X_4$ is unirational.  

Furthermore, if $k$ is a $C_r$ field, $n\geq 3$, and $s+1 > 2^r$ where $s = \max\{n-h,h\}$ a general quartic hypersurface $X_4\subset\mathbb{P}^{n+1}$ having double points along an $h$-plane is unirational.
\end{thm}
\begin{proof}
The case $h = n-1$ comes from Proposition \ref{Quart_2}. Let $\pi_{H}:X_4\dasharrow\mathbb{P}^{n-h-1}$ be the projection from $H = \left\langle p,\Lambda\right\rangle$, $\widetilde{\pi}_{H}:\widetilde{X}_4\rightarrow\mathbb{P}^{n-h-1}$, and $F_p\cong \mathbb{P}^{n-h-1}$ the fiber over $p$ of the blow-up $\widetilde{X}_4\rightarrow X_4$ along $H\cap X_4$. 

The generic fiber $\widetilde{X}_{4,\eta}$ of $\widetilde{\pi}_{H}:\widetilde{X}_4\rightarrow\mathbb{P}^{n-h-1}$ is a quartic hypersurface $\widetilde{X}_{4,\eta}\subset\mathbb{P}^{h+2}_{k(t_1,\dots,t_{n-h-1})}$ with double points along an $h$-plane and with a $k(t_1,\dots,t_{n-h-1})$-rational point induced by $F_p$. 

Therefore, Proposition \ref{Quart_2} yields that $\widetilde{X}_{4,\eta}$ is unirational over $k(t_1,\dots,t_{n-h-1})$ and hence $\widetilde{X}_{4}$ is unirational. 

Now, assume $k$ to be $C_r$. The exceptional divisor $\widetilde{E}$ in Lemma \ref{hyp} is a divisor of bidegree $(2,2)$ in $\mathbb{P}^{n-h}\times\mathbb{P}^h$, and since $X_4$ is general $\widetilde{E}$ maps dominantly onto $\mathbb{P}^{n-h}$. Furthermore, since both the fibrations $\widetilde{E}$ have quadric hypersurfaces as fibers and $s+1 > 2^r$ by Remark \ref{lang} the divisor $\widetilde{E}$ has a point. To conclude it is enough to apply \cite[Theorem 1.8]{Mas22} and Proposition \ref{Enr}.
\end{proof}

\begin{Remark}
In the second part of Theorem \ref{Quart_th} the assumption on the base field is needed in order to ensure the existence of a point in the exceptional divisor $\widetilde{E}\subset \mathbb{P}^{n-h}\times\mathbb{P}^h$. Indeed, over non $C_r$ fields there are divisors of bidegree $(2,2)$ in $\mathbb{P}^{n-h}\times\mathbb{P}^h$ without points \cite[Remark 4.15]{Mas22}. 
\end{Remark}

\begin{Remark}\label{RP_Nec}
In Theorem \ref{Quart_th} the assumption on the existence of a point in $X_4\setminus\Lambda$ is necessary as the following argument shows. Take for instance $n = 3$, $h = 1$ and $k = \mathbb{R}$. By Lemma \ref{hyp} $X_4$ is birational to a conic bundle $D\subset\mathcal{T}_{1,0,0}$ of multidegree $(4,3,3,2,2,2)$. We may write $D$ as the zero set of a polynomial 
$$
f = a_0x_0^4 + a_1x_0^3x_1+a_2x_0^2x_1^2+a_3x_0x_1^3+a_4x_1^4
$$
where $a_i\in\mathbb{R}[x_1,x_2,y_0,y_1,y_2]$. Let $V$ be the $\mathbb{R}$-vector space parametrizing these conic bundles and consider the map
$$
\begin{array}{cccc}
 ev: & V\times\mathbb{R}^5 & \longrightarrow & \mathbb{R}^5 \\ 
 & (f,(\overline{x}_1,\overline{x}_2,\overline{y}_0,\overline{y}_1,\overline{y}_2)) & \mapsto & P_{f,\overline{x}_1,\overline{x}_2,\overline{y}_0,\overline{y}_1,\overline{y}_2}(x_0) 
 \end{array}  
$$  
where $P_{f,\overline{x}_1,\overline{x}_2,\overline{y}_0,\overline{y}_1,\overline{y}_2}(x_0) = f(x_0,\overline{x}_1,\overline{x}_2,\overline{y}_0,\overline{y}_1,\overline{y}_2)\in\mathbb{R}[x_0]$. The polynomials $P_{f,\overline{x}_1,\overline{x}_2,\overline{y}_0,\overline{y}_1,\overline{y}_2}(x_0)$ with a real root form a semi-algebraic subset $R_{rr}\subset\mathbb{R}^5$ of maximal dimension, see for instance \cite[Section 4.2.3]{BPR06}. Hence, $ev^{-1}(R_{rr})\subset V\times\mathbb{R}^5$ is also a semi-algebraic subset of maximal dimension and then the Tarski-Seidenberg principle \cite[Section 2.2]{BCR98} yields that $Z_{rp}= \pi_1(ev^{-1}(R_{rr}))\subset V$, where $\pi_1:V\times\mathbb{R}^5\rightarrow V$ is the projection, is a semi-algebraic subset of maximal dimension. Note that
$$
Z_{rp} = \{D\subset\mathcal{T}_{1,0,0} \text{ of mutidegree } (4,3,3,2,2,2) \text{ having a rational point}\}.
$$
The complementary set $Z_{nrp} = Z_{rp}^c$ is non-empty, take for instance 
$$
D = \{(x_0^2+x_1^2+x_2^2)^2y_0^2+(x_0^2+x_1^2+2x_2^2)y_1^2+(x_0^2+x_1^2+3x_2^2)y_2^2 = 0\}\subset \mathcal{T}_{1,0,0}.
$$ 
Hence, $Z_{nrp}\subset V$ is also a semi-algebraic subset of maximal dimension. By Lemma \ref{hyp} the quartics $X_4\subset\mathbb{P}^4$ corresponding to the conic bundles in $Z_{nrp}$ do not have a point in $X_4\setminus\Lambda$ and hence they can not be unirational. 
\end{Remark}

\begin{Lemma}\label{sing0}
Consider the hypersurface 
$$
X_d = \{z_0^{d-2}Q + z_0^{d-3}z_1 A_1+\dots + z_0z_1^{d-3}A_{d-3}+z_1^{d-2}A_{d-2} = 0\}\subset\mathbb{P}^{n+1}_{(z_0,\dots,z_{n+1})}
$$
where $A_i = A_i^2 + x_0A_i^1+x_0^2A_i^0$, and $A_i^1\in k[z_1,\dots,z_{n+1}]_1$, $Q,A_i^2\in k[z_1,\dots,z_{n+1}]_2$, $A_i^0\in k$ for $i = 1,\dots,d-2$. Then $X_d$ is birational to the divisor
$$
Y_{(d-2,2)} = \left\{ x_0^{d-2}Q+\sum_{i=1}^{d-2}x_0^{d-2-i}x_1^iA_i = 0\right\}\subset\mathbb{P}^1_{(x_0,x_1)}\times \mathbb{P}^{n}_{(z_1,\dots, z_{n+1})}
$$
of bidegree $(d-2,2)$.
\end{Lemma}
\begin{proof}
It is enough to substitute $z_0 = \frac{x_0}{x_1}z_1$ in the equation of $X_d$ and to clear the denominators of the resulting polynomial. 
\end{proof}

For quartic hypersurfaces with a double point we get a version of Theorem \ref{Quart_th} with a slightly less restrictive condition on $n$. 

\begin{Proposition}\label{sing1}
Let $X_d\subset\mathbb{P}^{n+1}$ be a hypersurface of degree $d$ having multiplicity $d-2$ along a codimension two linear subspace $\Lambda\subset\mathbb{P}^{n+1}$, with a point $p\in X_d\setminus\Lambda$, a double point $q\in X_d\setminus\Lambda$ and otherwise general. If either $d = 4$ and $n\geq 2$ or $d = 5$ and $n\geq 4$ then $X_d$ is unirational.  
\end{Proposition}
\begin{proof}
The equation of $X_d\subset\mathbb{P}^{n+1}$ can be written as in Lemma \ref{sing0}. First, consider the case $d = 4$. By Lemma \ref{sing0} $X_4$ is birational to a divisor $Y_{(2,2)}\subset \mathbb{P}^1_{(x_0,x_1)}\times \mathbb{P}^{n}_{(z_1,\dots, z_{n+1})}$ of bidegree $(2,2)$. The point $p\in X_d\setminus\Lambda$ yields a point $p'\in Y_{(2,2)}$. Let $H$ be a general $2$-plane in $\mathbb{P}^{n}_{(z_1,\dots, z_{n+1})}$ through the projection of $p'$ and set $\Gamma = \mathbb{P}^1_{(x_0,x_1)}\times H$. Then $W = Y_{(2,2)}\cap\Gamma$ is a divisor of bidegree $(2,2)$ in $\mathbb{P}^1_{(x_0,x_1)} \times\mathbb{P}^2$ which by \cite[Corollary 8]{KM17} in unirational. Since $W$ maps dominantly onto $\mathbb{P}^1_{(x_0,x_1)}$ to conclude it is enough to apply Proposition \ref{Enr}.

Now, let $d = 5$. Then by Lemma \ref{sing0} $X_5$ is birational to a divisor $Y_{(3,2)}\subset \mathbb{P}^1_{(x_0,x_1)}\times \mathbb{P}^{n}_{(z_1,\dots, z_{n+1})}$ of bidegree $(3,2)$ and arguing as in the proof of Proposition \ref{p0} we get that $Y_{(3,2)}$ is unirational. 
\end{proof}

\begin{Lemma}\label{1a}
Let $D\subset\mathbb{P}^{n-h}\times\mathbb{P}^h$ be divisor of bidegree $(1,a)$. Then $D$ is unirational. 
\end{Lemma}
\begin{proof}
If $n-h = 1$ then the projection $\pi_{2|D}:D\rightarrow\mathbb{P}^h$ is birational. Assume $n-h\geq 2$ and consider the following incidence variety
\[
  \begin{tikzpicture}[xscale=1.5,yscale=-1.5]
    \node (A0_1) at (1, 0) {$W = \{(z,L)\: | \: z\in \pi_1^{-1}(L)\}\subseteq D\times \mathbb{G}(1,n-h)$};
    \node (A1_0) at (0, 1) {$D$};
    \node (A1_2) at (2, 1) {$\mathbb{G}(1,n-h)$};
    \path (A0_1) edge [->]node [auto] {$\scriptstyle{\beta}$} (A1_2);
    \path (A0_1) edge [->]node [auto,swap] {$\scriptstyle{\alpha}$} (A1_0);
  \end{tikzpicture}
  \]
where $\mathbb{G}(1,n-h)$ is the Grassmannian parametrizing lines in $\mathbb{P}^{n-h}$. The generic fiber $W_{\eta}$ of $\alpha:W\rightarrow\mathbb{G}(1,n-h)$ is a divisor $W_{\eta}$ of bidegree $(1,a)$ in $\mathbb{P}^1_{k(t_1,\dots,t_{2(n-h-1)})}\times \mathbb{P}^h$, and hence $W_{\eta}$ is rational over $k(t_1,\dots,t_{2(n-h-1)})$. Therefore, $W$ is rational and since $\beta:W\rightarrow D$ is dominant we get that $D$ is unirational. 
\end{proof}

The following is our main result on the unirationality of quartic hypersurface containing a linear subspace.

\begin{thm}\label{Quart_th2}
Let $X_4\subset\mathbb{P}^{n+1}$ be an irreducible quartic hypersurface containing an $h$-plane $\Lambda\subset\mathbb{P}^{n+1}$. Assume that $\Lambda\cap\Sing(X_4)$ has dimension at most $h-2$ and that $X_4$ is not a cone over a smaller dimensional quartic. If $n\geq 3$ and $h\geq 2$ then $X_4$ is unirational.  
\end{thm}
\begin{proof}
Consider the blow-up $\widetilde{X}_4$ along $\Lambda$ in Lemma \ref{hyp}. Under our hypotheses $X_4$ can not be singular at the general point of $\Lambda$. So the exceptional divisor $\widetilde{E}\subset\mathbb{P}^{n-h}\times\mathbb{P}^h$ has bidegree $(1,3)$, and Lemma \ref{1a} yields that $\widetilde{E}$ is unirational. The projection $\pi:\widetilde{X}_4\rightarrow\mathbb{P}^{n-h}$ is a fibration in $h$-dimensional cubic hypersurfaces. Therefore, thanks to Proposition \ref{Enr} and Remark \ref{QC} to conclude it is enough to show that $\pi_{|\widetilde{E}}:\widetilde{E}\rightarrow \mathbb{P}^{n-h}$ is dominant.

Since $X_4$ is not a cone over a smaller dimensional quartic by the expression of $\widetilde{X}_4$ in (\ref{st}) we see that $\pi_{|\widetilde{E}}:\widetilde{E}\rightarrow \mathbb{P}^{n-h}$ is not dominant if and only if the equation of $X_d$ is of the following form 
$$
X_d = \left\{\sum_{i=0}^{n-h}z_i(C_i(z_0,\dots,z_{n-h})+c_iP) = 0\right\}\subset\mathbb{P}^{n+1}
$$
where $C_i\in k[z_0,\dots,z_{n-h}]_{d-1}$, $c_i\in k$ and $P\in k[z_0,\dots,z_{n+1}]_{d-1}$. In particular,
$$Z = \{z_0 = \dots = z_{n-h} = P = 0\}\subset \Lambda\cap\Sing(X_d)$$
and since $\dim(Z)\geq h-1$ we get a contradiction. 
\end{proof}

\begin{Corollary}\label{quartsmooth}
Assume that $h\geq 2$, $n > 2h-1$ and let $X_4\subset\mathbb{P}^{n+1}$ be a smooth quartic hypersurface containing an $h$-plane. Then $X_4$ is unirational. 
\end{Corollary}
\begin{proof}
Note that $n > 2h-1$ yields that a general quartic containing a fixed $h$-plane is smooth. The claim follows from Theorem \ref{Quart_th2}. 
\end{proof}

\section{Divisors in products of projective spaces and Quintics}\label{sec3}

We now study the unirationality of quintic hypersurfaces that are singular along linear subspaces and of divisors of bidegree $(3,2)$ in products of projective spaces. We begin with a series of preliminary results that we will need later on. 

\begin{Lemma}\label{LCub}
Let $\mathcal{C}\rightarrow W$ be a fibration in $m$-dimensional cubic hypersurfaces with $m\geq 2$ and $W$ a rational variety over a $C_r$ field $k$. If the general fiber of $\mathcal{C}\rightarrow W$ does not have a triple point and
$$
m > 3^{r+\dim(W)} - 2
$$
then $\mathcal{C}$ is unirational.
\end{Lemma}
\begin{proof}
The variety $\mathcal{C}$ is birational to an $m$-dimensional cubic hypersurface $\mathcal{C}'$ over the function field $F = k(t_1,\dots,t_{\dim(W)})$. Since $m+2 > 3^{r+\dim(W)}$ Remark \ref{lang} yields that $\mathcal{C}'$ has a point. Hence, by \cite[Theorem 1]{Kol02} $\mathcal{C}'$ is unirational over $F$ and so $\mathcal{C}$ is unirational.
\end{proof}

\begin{Lemma}\label{Q_LS}
Let $Q^{N-1}\subset\mathbb{P}^{N}$ be a smooth $(N-1)$-dimensional quadric hypersurface over a $C_r$ field. If 
$$
N-2s+1 > 2^r 
$$
for some $0\leq s \leq \lfloor\frac{N-1}{2}\rfloor$ then through any point of $Q^{N-1}$ there is an $s$-plane contained in $Q^{N-1}$.
\end{Lemma}
\begin{proof}
Since $N-1 > 2^r - 2$ by Remark \ref{lang} $Q^{N-1}$ has a point. Hence $Q^{N-1}$ is rational and in particular the set of its rational points in dense. Take a point $x_0 \in Q^{N-1}$. 

The tangent space $T_{x_0}Q^{N-1}$ cuts out on $Q^{N-1}$ a cone with vertex $x_0$ over an $(N-3)$-dimensional quadric $Q^{N-3}$. Since $N-3 > 2^r - 2$ Remark \ref{lang} yields the existence of a point $x_1 \in Q^{N-3}$. The line $\left\langle x_0,x_1\right\rangle$ is therefore contained in $Q^{N-1}$. 

Now, $T_{x_1}Q^{N-3}\cap Q^{N-3}$ is a cone with vertex $x_1$ over an $(N-5)$-dimensional quadric $Q^{N-5}$ which again by Remark \ref{lang} has a point $x_2\in Q^{N-5}$. The $2$-plane $\left\langle x_0,x_1,x_2\right\rangle$ is contained in $Q^{N-1}$.

Proceeding recursively we have that $T_{x_{s-1}}Q^{N-2(s-1)-1}\cap Q^{N-2(s-1)-1}$ is a cone with vertex $x_{s-1}$ over an $(N-2s-1)$-dimensional quadric $Q^{N-2s-1}$. Since $N-2s+1 > 2^r$ the quadric $Q^{N-2s-1}$ has a point $x_s$ and the $s$-plane $\left\langle x_0,\dots,x_s\right\rangle$ is then contained in $Q^{N-1}$. 
\end{proof}

\begin{Lemma}\label{ciq}
Let $X = Q_1\cap\dots \cap Q_c\subset\mathbb{P}^N$ be a smooth complete intersection of quadrics. Assume that $X$ contains a $(c-1)$-plane $\Lambda$. Then $X$ is rational. 
\end{Lemma}
\begin{proof}
Note that since $X$ is smooth $Q_i$ must be smooth along $X$. Let $H$ be a general $c$-plane containing $\Lambda$. Then $H$ intersects $Q_i$ along $\Lambda\cup H_i$ where $H_i$ is a $(c-1)$-plane. Hence we get $c$ linear subspaces of dimension $c-1$ of $H\cong \mathbb{P}^c$. 

Since $H$ is general these $(c-1)$-planes intersect in a point $x_H = H_1\cap\dots \cap H_c \in X$. To conclude it is enough to parametrize $X$ with the $\mathbb{P}^{N-c}$ of $c$-planes containing $\Lambda$. 
\end{proof}

\begin{Proposition}\label{ciqls}
Let $X = Q_1\cap\dots \cap Q_c\subset\mathbb{P}^{N}$ be a smooth complete intersection of quadrics over a $C_r$ field. If
$$
N-s(c+1)+1 > 2^rc
$$
then through any point of $X$ there is an $s$-plane contained in $X$.
\end{Proposition}
\begin{proof}
For $s = 0$ the claim follows from Remark \ref{lang}. We proceed by induction on $s$. Since
$$
N-(s-1)(c+1)+1 > N-s(c+1)+1 > 2^rc
$$
through any point of $X$ there is an $(s-1)$-plane. Fix one of these $(s-1)$-planes and denote it by $\Lambda^{s-1}\subset X$. The $s$-planes in $\mathbb{P}^{N}$ containing $\Lambda^{s-1}$ are parametrized by $\mathbb{P}^{N-s}$. 

Arguing as in the proof of Lemma \ref{Q_LS} we see that requiring such an $s$-plane to be contained in $Q_i$ yields $s$ linear equations given by the tangent spaces of $Q_i$ at $s$ general points of $\Lambda^{s-1}$ plus a quadratic equation induced by $Q_i$ itself. 

Hence, the points of $\mathbb{P}^{N-s}$ corresponding to $s$-planes contained in $X$ are parametrized by a subvariety cut out by $sc$ linear equations and $c$ quadratic equations that is an intersection $Y$ of $c$ quadrics in $\mathbb{P}^{N-s-sc}$. Since $\dim(Y) = N-s(c+1)-c > 2^rc-1-c \geq -1$ and $N-s(c+1)+1 > 2^rc$ Remark \ref{lang} yields that $Y$ has a point and so there is an $s$-plane through $\Lambda^{s-1}$ contained in $X$.     
\end{proof}

\begin{Corollary}\label{rat_ciq}
Let $X = Q_1\cap\dots \cap Q_c\subset\mathbb{P}^{N}$ be a smooth complete intersection of quadrics over a $C_r$ field. If
$$
N-c^2+2 > 2^rc
$$
then $X$ is rational.
\end{Corollary}
\begin{proof}
By Proposition \ref{ciqls} with $s = c-1$ the complete intersection $X$ contains a $(c-1)$-plane and hence the claim follows from Lemma \ref{ciq}.
\end{proof}

We are now ready to prove our first result on unirationality of quintic hypersurfaces. 

\begin{Proposition}\label{quint3}
Let $X^{n}_5 = \{\overline{f} = 0\}\subset\mathbb{P}^{n+1}$ be a quintic hypersurface of the form
$$
\overline{f} = \sum_{ i_0+\dots +i_{n-h} = 3}x_0^{i_0}\dots x_{n-h}^{i_{n-h}}A_{i_0,\dots,i_{n-h}}(y_0,\dots,y_h,x_{n-h}) 
$$
with $A_{i_0,\dots,i_{n-h}}$ general quadratic polynomials. If either 
\begin{itemize}
\item[(i)] $n\geq 4$, $2h\geq n+3$ and the complete intersection
$$
W' = \{A_{3,\dots,3} = A_{2,1,0,\dots,0} = x_1 = \dots = x_{n-h} = 0\}\subset\mathbb{P}^{h}_{(y_0,\dots,y_h)}
$$
contains a line; or
\item[(ii)] $h\geq 5$, $n\geq 6$ and the quadric 
$$
\widetilde{Q} = \{x_0 = \dots = x_{n-h} = A_{3,0,\dots,0} = 0\}\subset\mathbb{P}^{h}
$$
contains a $2$-plane;
\end{itemize}
then $X_5^{n}$ is unirational. 
\end{Proposition}
\begin{proof}
Note that $X^{n}_5$ is a hypersurface of the form (\ref{pol_proj}) in Proposition \ref{qb-hyp}. Hence, $X^{n}_5$ has multiplicity three along $\Lambda = \{x_0 = \dots = x_{n-h} = 0\}$, and multiplicity two along $\Lambda' = \{y_0 = \dots = y_h = x_{n-h} = 0\}$. 

First, consider $\text{(i)}$. Fix a point $p \in \Lambda'$, say $p = [1:0:\dots : 0]$. The lines through $p$ intersecting $X_5^n$ at $p$ with multiplicity at least four are parametrized by the variety
$$Y = \{A_{3,\dots,3} = x_1A_{2,1,0,\dots,0} + \dots + x_{n-h}A_{2,0,\dots,0,1} = 0\}\subset\mathbb{P}^{n}_{(x_1,\dots,x_{n-h},y_0,\dots,y_h)}.$$ 

Assume $Y$ to be unirational. Associating to a general point $y\in Y$, representing a line $l_y$, the fifth intersection point of $l_y$ and $X_5^n$ we get a rational map
$$
\psi:Y\dasharrow X_5^n.
$$
Set $Y_{\psi} = \overline{\psi(Y)}\subset X_5^n$, and let $\pi:X_5^n\dasharrow \mathbb{P}^{n-h}$ be the restriction to $X_5^n$ of the projection from $\Lambda$. We want to prove that $\pi_{|Y_{\psi}}:Y_{\psi}\dasharrow \mathbb{P}^{n-h}$ is dominant. Indeed, if so $Y_{\psi}$ would be a unirational variety transverse to the quadric fibration induced by the projection from $\Lambda$ and dominating $\mathbb{P}^{n-h}$, and Proposition \ref{Enr} would imply that $X_5^n$ is unirational. 

Since $Y_{\psi}$ is unirational it is enough to prove that the induced map $\overline{\pi}_{|Y_{\psi}}\overline{Y}_{\psi}\dasharrow \mathbb{P}^{n-h}$ between the algebraic closure is dominant. Consider the complete intersection 
$$
Z = Y\cap X^n_5 = \{A_{3,\dots,3} = x_1A_{2,1,0,\dots,0} + \dots + x_{n-h}A_{2,0,\dots,0,1} = \overline{f}= 0\}\subset \mathbb{P}^{n}_{(x_1,\dots,x_{n-h},y_0,\dots,y_h)}.
$$
A general point of $z\in Z\subset X^n_5$ represents a line $l_z$ intersecting $X^n_5$ with multiplicity four at $p$ and multiplicity one at $z$. Hence $Z\subset Y_{\psi}$.

Fix a general point $q\in\mathbb{P}^{n-h}$ and consider 
$$Z_q = \overline{\pi}_{|Y_{\psi}}^{-1}(q)\cap Z.$$
Note that since $h > 2^{r+2}>3$ we have that $Z_q$ is non empty. So $\overline{\pi}_{|W_{\psi}}$ is dominant when restricted to $\overline{Z}$, and hence it is dominant.

Now, we prove that $Y$ is unirational. In $\mathbb{P}^{n}_{(x_1,\dots,x_{n-h},y_0,\dots,y_h)}$ fix the point $p' = [1:0:\dots:0]$. The quadric 
$$\{x_1 = \dots = x_{n-h} = A_{2,1,0,\dots,0}\}\subset\mathbb{P}^{h}_{(y_0,\dots,y_h)}$$ parametrizes lines that are contained in 
$$Y_3 = \{x_1A_{2,1,0,\dots,0} + \dots + x_{n-h}A_{2,0,\dots,0,1}\}.$$
Indeed, the tangent cone of $Y_3$ at $p'$ is defined by $\{A_{2,1,0,\dots,0} = 0\}$ and $\{x_1 = \dots = x_{n-h}= 0\}\subset Y_3$. So these lines intersect $Y_3$ with multiplicity three at $p'$ and at another point in the linear space $\{x_1 = \dots = x_{n-h}= 0\}$. 

Set $Y_2 = \{A_{3,\dots,3} = 0\}$, and let $W$ be the cone over $W'$ with vertex $p'$. Since $p'$ is in the vertex of $Y_2$ we have that
$$W\subset Y = Y_2\cap Y_3.$$
Furthermore, since $W'$ contains a line by Lemma \ref{ciq} it is rational, and hence $W$ is rational as well. 

As in the proof of Proposition \ref{p_bun} we construct a quadric bundle $\widetilde{\mathcal{Q}}\rightarrow W$ with $(n-4)$-dimensional fibers whose general point $(w,l_w)$ represents a point $w\in W$ and a line $l_w$ which is contained in $Y_2$ and intersects $Y_3$ with multiplicity two at $w$. 

Associating to $(w,l_w)$ the third point of intersection of $l_w$ and $Y_3$ we get a rational map $W\dasharrow Y$ and arguing as in the proof of Proposition \ref{p_bun} we see that such rational map is dominant. 

Now, we consider $\text{(ii)}$. Note that 
$$
\widetilde{Q}' = \{x_0 = \dots = x_{n-h} = A_{3,0,\dots,0} = 0\}\subset X^n_5\subset\mathbb{P}^{n+1}
$$ 
is an $h$-dimensional quadric cone over $\widetilde{Q}$ with vertex $[1:0:\dots:0]$. Since $\widetilde{Q}$ contains a $2$-plane $\widetilde{Q}'$ contains a $3$-plane $H\subset X^n_5$. If $x\in H$ is a general point the lines through $x$ intersecting $X^n_5$ with multiplicity four at $x$ are parametrized by a complete intersection of a quadric and a cubic in $\mathbb{P}^{n-1}$. Hence, we get a fibration $\mathcal{Y}_{2,3}\rightarrow H$ whose fiber over a general $x\in H$ is a complete intersection $Y_{2,3,x}$ of a quadric $Y_{2,x}\subset\mathbb{P}^{n-1}$ and a cubic $Y_{3,x}\subset\mathbb{P}^{n-1}$. 

The generic fiber of $\mathcal{Y}_{2,3}\rightarrow H$ is then a complete intersection $\mathcal{Y}_{2,3,k(H)} = \mathcal{Y}_{2,k(H)}\cap \mathcal{Y}_{3,k(H)}$ satisfying the hypotheses of Lemma \ref{Enr_gen}. Indeed, by considering the $2$-plane parametrizing lines in $H$ through a general $x\in H$ we get a $2$-plane over $k(H)$ contained in $\mathcal{Y}_{2,k(H)}$.      

Therefore, by Lemma \ref{Enr_gen} $\mathcal{Y}_{2,3,k(H)}$ is unirational over $k(H)$ and so $\mathcal{Y}_{2,3}$ is unirational. Now, a general point of $\mathcal{Y}_{2,3}$ represents a pair $(x,l_x)$ with $x\in H$ general and $l_x$ a line intersecting $X_5^n$ with multiplicity four at $x$. As usual associating to $(x,l_x)\in \mathcal{Y}_{2,3}$ the fifth point of intersection of $l_x$ with $X_5^n$ we get a rational map
$$
\phi:\mathcal{Y}_{2,3}\dasharrow X_5^n.
$$  
Note that $\dim(\mathcal{Y}_{2,3}) = 3+(n-3) = n$. Arguing as in the last part of the proof of Proposition \ref{p_bun} we see that $\phi$ is generically finite and therefore it is dominant. 
\end{proof}

\begin{thm}\label{333}
Let $D\subset \mathbb{P}^{n-h}\times\mathbb{P}^{h+1}$, with $n-h\geq 2$, be a general divisor of bidegree $(3,2)$ over a $C_r$ field. If either
\begin{itemize}
\item[(i)] $n\geq 4$, $2h\geq n+3$ and $h > 2^{r+1}+2$; or
\item[(ii)] $n\geq 6$, $h\geq 5$ and $h > 2^{r}+3$; or
\item[(iii)] $n-h-1 > 3^{r+1}-2$;
\end{itemize}
then $D$ is unirational. 
\end{thm}
\begin{proof}
By Proposition \ref{qb-hyp} $D$ is birational to a quintic hypersurface $X^n_{5}\subset\mathbb{P}^{n+1}$ of the form considered in Proposition \ref{quint3}. It is straightforward to see that if $D \subset \mathbb{P}^{n-h}\times\mathbb{P}^{h+1}$ is a general divisor of bidegree $(3,2)$ then the $A_{i_0,\dots,i_{n-h}}$ are general quadratic polynomials in the variables $y_0,\dots,y_{h},x_{n-h}$.

Consider $\text{(i)}$. Since, $h > 2^{r+1}+2$ by Proposition \ref{ciqls} the complete intersection $W'$ in Proposition \ref{quint3} contains a line, and hence Proposition \ref{quint3} yields that $X^n_{5}$ is unirational. 

For $\text{(ii)}$ note that since $h > 2^{r}+3$ Lemma \ref{Q_LS} implies that the quadric $\widetilde{Q}$ in Proposition \ref{quint3} contains a $2$-plane. 

Now, consider $\text{(iii)}$. The projection $\pi_2:D\rightarrow\mathbb{P}^{h+1}$ endows $D$ with a structure of fibration in $(n-h-1)$-dimensional cubic hypersurfaces. Take a general line $L\subset \mathbb{P}^{h+1}$ and set $\mathcal{C}_L = \pi_2^{-1}(L)$. Then $\mathcal{C}_L$ is a fibration in $(n-h-1)$-dimensional cubic hypersurfaces over $\mathbb{P}^1$ and since $n-h-1 > 3^{r+1}-2$ Lemma \ref{LCub} yields that $\mathcal{C}_L$ is unirational. Finally, to conclude it is enough to note that $\pi_{|\mathcal{C}_L}:\mathcal{C}_L\rightarrow\mathbb{P}^{n-h}$ is dominant and to apply Proposition \ref{Enr}.
\end{proof}

\begin{Remark}
Take for instance $r = 0$ that is $k$ is algebraically closed. Remark \ref{lang} gives the rationality of a $D\subset \mathbb{P}^{n-h}\times\mathbb{P}^{h+1}$ as in Theorem \ref{333} for $h > 2^{n-h}-2$ while Theorem \ref{333} gives the unirationality of $D$ for $h > 4$ as long as $n\geq 4$ and $h \geq n-h+3$. For example take $h = 10$. Then Remark \ref{lang} yields that $D$ is rational for $n\leq 13$ while Theorem \ref{333} gives the unirationality of $D$ for $n\leq 17$. 

Furthermore, there are cases covered by $\text{(iii)}$ but not by $\text{(i)}$. For instance, by $\text{(iii)}$ we get that a general $D\subset \mathbb{P}^{n-1}\times\mathbb{P}^{2}$ of bidegree $(3,2)$ and dimension at least four is unirational.  

For $r\geq 1$ $\text{(i)}$ generally performs better that $\text{(ii)}$. For instance, the case $h = 7, n = 11$ is covered by $\text{(i)}$ but not by $\text{(ii)}$. 
\end{Remark}

\begin{Remark}\label{sr333}
In particular, Theorem \ref{333} (ii) together with \cite[Theorem A]{ABP18} yields that a very general divisor of bidegree $(3,2)$ in $\mathbb{P}^3\times\mathbb{P}^2$, over an algebraically closed field of characteristic zero, is unirational but not stably rational. 
\end{Remark}

We end this section with our main results on the unirationality of quintic hypersurfaces which are singular along a linear subspace. 

\begin{thm}\label{th5}
Let $X_5\subset \mathbb{P}^{n+1}$ be a quintic hypersurface over a $C_r$ field having multiplicity three along an $h$-plane and otherwise general. Assume that $n-h \geq 2$. If either
\begin{itemize}
\item[(i)] $n\geq 5$, $2h\geq n+4$ and $h > 2^{r+1}+3$; or
\item[(ii)] $n\geq 7$, $h\geq 6$ and $h > 2^{r}+4$; or
\item[(iii)] $n-h-1 > 3^{r+1}-2$;
\end{itemize}
then $X_5$ is unirational. 

Similarly, if $X_5\subset \mathbb{P}^{n+1}$ has multiplicity two along an $h$-plane with $h\geq 2$ and is otherwise general, and either
\begin{itemize}
\item[(i)] $n\geq 5$, $2h\leq n-4$ and $n-h > 2^{r+1}+3$; or
\item[(ii)] $n\geq 7$, $n-h-1\geq 5$ and $n-h-1 > 2^{r}+3$; or
\item[(iii)] $h > 3^{r+1}-1$;
\end{itemize}
then $X_5$ is unirational. 
\end{thm}
\begin{proof}
By Lemma \ref{hyp} the exceptional divisor $\widetilde{E}\subset\widetilde{X}_5$ is a divisor of bidegree $(3,2)$ in $\mathbb{P}^{n-h}\times\mathbb{P}^h$. Since $X_5$ is general $\widetilde{E}$ maps dominantly onto $\mathbb{P}^{n-h}$, and hence to conclude we apply Proposition \ref{Enr} and Theorem \ref{333}. For the second statement it is enough to  argue as in the previous case on a divisor of bidegree $(2,3)$ in $\mathbb{P}^{n-h}\times\mathbb{P}^h$.    
\end{proof}

\begin{Proposition}\label{n-h=1}
Let $X_5\subset \mathbb{P}^{n+1}$ be a quintic hypersurface over a field $k$ having multiplicity three along an $(n-1)$-plane and otherwise general. If either 
\begin{itemize}
\item[(i)] $n\geq 5$ and $k$ is either a number field or a real closed field; or
\item[(ii)] $k$ is $C_r$, $n > 4$ and $n > 2^r$;
\end{itemize}
then $X_5$ is unirational. 
\end{Proposition}
\begin{proof}
The exceptional divisor $\widetilde{E}\subset\widetilde{X}_5$ is a general divisor of bidegree $(3,2)$ in $\mathbb{P}^{1}\times\mathbb{P}^{n-1}$. By \cite[Corollary 4.13, Lemma 4.18]{Mas22} and Remark \ref{lang} under our hypotheses $\widetilde{E}$ has a point and hence \cite[Theorem 1.8]{Mas22} yields that $\widetilde{E}$ is unirational. To conclude it is enough to argue as in the proof of Theorem \ref{th5} applying Propositions \ref{Enr}.  
\end{proof}

\begin{Remark}\label{kol_rcf}
When $k$ is a real closed field the unirationality of a quintic hypersurface $X_5\subset \mathbb{P}^{n+1}$ as in Proposition \ref{n-h=1} follows from \cite[Corollary 1.8]{Kol99}.
\end{Remark}

\begin{Proposition}\label{d_gen}
Let $X_d\subset \mathbb{P}^{n+1}$ be a hypersurface of degree $d$ over a field $k$ having multiplicity $d-2$ along an $h$-plane and otherwise general. Assume that $(h+1)(d-2)$ is odd. If either 
\begin{itemize}
\item[(i)] $d\leq\frac{5h+3}{h+1}$; or
\item[(ii)] $d\leq \frac{6h+1}{h+1}$, $h\leq 4$, $k$ is $C_r$ and $h+1 > 2^{r+n-h-1}$;
\end{itemize}
then $X_5$ is unirational. 
\end{Proposition}
\begin{proof}
In this case the exceptional divisor $\widetilde{E}\subset\widetilde{X}_d$ is a general divisor of bidegree $(d-2,2)$ in $\mathbb{P}^{n-h}\times\mathbb{P}^{h}$. Note that the discriminant of the quadric bundle $\widetilde{E}\rightarrow \mathbb{P}^{n-h}$ has degree $(h+1)(d-2)$. Hence the claim follows from \cite[Theorem 1.7]{Mas22} and Propositions \ref{Enr}.  
\end{proof}

\bibliographystyle{amsalpha}
\bibliography{Biblio}

\providecommand{\bysame}{\leavevmode\hbox to3em{\hrulefill}\thinspace}
\providecommand{\MR}{\relax\ifhmode\unskip\space\fi MR }
\providecommand{\MRhref}[2]{%
  \href{http://www.ams.org/mathscinet-getitem?mr=#1}{#2}
}
\providecommand{\href}[2]{#2}
\begin{thebibliography}{BCTSSD85}

\bibitem[ABP18]{ABP18}
A.~Auel, C.~B\"{o}hning, and A.~Pirutka, \emph{Stable rationality of quadric
  and cubic surface bundle fourfolds}, Eur. J. Math. \textbf{4} (2018), no.~3,
  732--760. \MR{3851115}

\bibitem[AO18]{AO18}
H.~Ahmadinezhad and T.~Okada, \emph{Stable rationality of higher dimensional
  conic bundles}, \'{E}pijournal G\'{e}om. Alg\'{e}brique \textbf{2} (2018),
  Art. 5, 13. \MR{3816900}

\bibitem[Ara05]{Ar05}
C.~Araujo, \emph{Rationally connected varieties}, Snowbird lectures in
  algebraic geometry, Contemp. Math., vol. 388, Amer. Math. Soc., Providence,
  RI, 2005, pp.~1--16. \MR{2182887}

\bibitem[BCR98]{BCR98}
J.~Bochnak, M.~Coste, and M.~F. Roy, \emph{Real algebraic geometry}, Ergebnisse
  der Mathematik und ihrer Grenzgebiete (3) [Results in Mathematics and Related
  Areas (3)], vol.~36, Springer-Verlag, Berlin, 1998, Translated from the 1987
  French original, Revised by the authors. \MR{1659509}

\bibitem[BCTSSD85]{BCSS85}
A.~Beauville, J.~L. Colliot-Th\'{e}l\`ene, J.~J. Sansuc, and
  P.~Swinnerton-Dyer, \emph{Vari\'{e}t\'{e}s stablement rationnelles non
  rationnelles}, Ann. of Math. (2) \textbf{121} (1985), no.~2, 283--318.
  \MR{786350}

\bibitem[BPR06]{BPR06}
S.~Basu, R.~Pollack, and M.~F. Roy, \emph{Algorithms in real algebraic
  geometry}, second ed., Algorithms and Computation in Mathematics, vol.~10,
  Springer-Verlag, Berlin, 2006. \MR{2248869}

\bibitem[BRS19]{BRS19}
M.~Bolognesi, F.~Russo, and G.~Staglian\`o, \emph{Some loci of rational cubic
  fourfolds}, Math. Ann. \textbf{373} (2019), no.~1-2, 165--190. \MR{3968870}

\bibitem[BvB18]{BG18}
C.~B\"{o}hning and H.~C.~Graf von Bothmer, \emph{On stable rationality of some
  conic bundles and moduli spaces of {P}rym curves}, Comment. Math. Helv.
  \textbf{93} (2018), no.~1, 133--155. \MR{3777127}

\bibitem[CG72]{CG72}
C.~H. Clemens and P.~A. Griffiths, \emph{The intermediate {J}acobian of the
  cubic threefold}, Ann. of Math. (2) \textbf{95} (1972), 281--356. \MR{302652}

\bibitem[Cil80]{Cil80}
C.~Ciliberto, \emph{Remarks on some classical unirationality theorems for
  hypersurfaces and projective algebraic complete intersections}, Ricerche Mat.
  \textbf{29} (1980), no.~2, 175--191. \MR{632207}

\bibitem[CMM07]{CMM07}
A.~Conte, M.~Marchisio, and J.~P. Murre, \emph{On the $k$-unirationality of the
  cubic complex}, Atti della Accademia Peloritana dei Pericolanti, Classe di
  Scienze FF.MM.NN \textbf{LXXXV} (2007), 1--6.

\bibitem[CTP16]{CP16}
J.~L. Colliot-Th\'{e}l\`ene and A.~Pirutka, \emph{Hypersurfaces quartiques de
  dimension 3: non-rationalit\'{e} stable}, Ann. Sci. \'{E}c. Norm. Sup\'{e}r.
  (4) \textbf{49} (2016), no.~2, 371--397. \MR{3481353}

\bibitem[CTSSD87]{CSS87}
J.~L. Colliot-Th\'{e}l\`ene, J.~J. Sansuc, and P.~Swinnerton-Dyer,
  \emph{Intersections of two quadrics and {C}h\^{a}telet surfaces. {I}}, J.
  Reine Angew. Math. \textbf{373} (1987), 37--107. \MR{870307}

\bibitem[dF13]{dF13}
T.~de~Fernex, \emph{Birationally rigid hypersurfaces}, Invent. Math.
  \textbf{192} (2013), no.~3, 533--566. \MR{3049929}

\bibitem[HKT16]{HKT16}
B.~Hassett, A.~Kresch, and Y.~Tschinkel, \emph{Stable rationality and conic
  bundles}, Math. Ann. \textbf{365} (2016), no.~3-4, 1201--1217. \MR{3521088}

\bibitem[HMP98]{HMP98}
J.~Harris, B.~Mazur, and R.~Pandharipande, \emph{Hypersurfaces of low degree},
  Duke Math. J. \textbf{95} (1998), no.~1, 125--160. \MR{1646558}

\bibitem[HPT18]{HPT18}
B.~Hassett, A.~Pirutka, and Y.~Tschinkel, \emph{Stable rationality of quadric
  surface bundles over surfaces}, Acta Math. \textbf{220} (2018), no.~2,
  341--365. \MR{3849287}

\bibitem[HPT19]{HPT19}
\bysame, \emph{A very general quartic double fourfold is not stably rational},
  Algebr. Geom. \textbf{6} (2019), no.~1, 64--75. \MR{3904799}

\bibitem[HT00]{HT00}
J.~Harris and Y.~Tschinkel, \emph{Rational points on quartics}, Duke Math. J.
  \textbf{104} (2000), no.~3, 477--500. \MR{1781480}

\bibitem[IM71]{IM71}
V.~A. Iskovskih and J.~I. Manin, \emph{Three-dimensional quartics and
  counterexamples to the {L}\"{u}roth problem}, Mat. Sb. (N.S.)
  \textbf{86(128)} (1971), 140--166. \MR{0291172}

\bibitem[IP99]{IP99}
V.~A. Iskovskikh and Y.~G. Prokhorov, \emph{Fano varieties}, Algebraic
  geometry, {V}, Encyclopaedia Math. Sci., vol.~47, Springer, Berlin, 1999,
  pp.~1--247. \MR{1668579}

\bibitem[KM17]{KM17}
J\'{a}nos Koll\'{a}r and Massimiliano Mella, \emph{Quadratic families of
  elliptic curves and unirationality of degree 1 conic bundles}, Amer. J. Math.
  \textbf{139} (2017), no.~4, 915--936. \MR{3689320}

\bibitem[Kol95]{Ko95}
J.~Koll\'{a}r, \emph{Nonrational hypersurfaces}, J. Amer. Math. Soc. \textbf{8}
  (1995), no.~1, 241--249. \MR{1273416}

\bibitem[Kol99]{Kol99}
\bysame, \emph{Rationally connected varieties over local fields}, Ann. of Math.
  (2) \textbf{150} (1999), no.~1, 357--367. \MR{1715330}

\bibitem[Kol02]{Kol02}
J\'{a}nos Koll\'{a}r, \emph{Unirationality of cubic hypersurfaces}, J. Inst.
  Math. Jussieu \textbf{1} (2002), no.~3, 467--476. \MR{1956057}

\bibitem[Mas22]{Mas22}
A.~Massarenti, \emph{On the unirationality of quadric bundles},
  \url{https://arxiv.org/abs/2204.08793}, 2022.

\bibitem[Mor36]{Mor36}
U.~Morin, \emph{Sulla unirazionalit\`a delle ipersuperficie algebriche del
  quarto ordine}, Rendiconti Acc. Naz. Lincei \textbf{24} (1936), 192--194.

\bibitem[Mor42]{Mo40}
\bysame, \emph{Sull'unirazionalit\`a dell'ipersuperficie algebrica di qualunque
  ordine e dimensione sufficientemente alta}, Atti {S}econdo {C}ongresso {U}n.
  {M}at. {I}tal., {B}ologna, 1940, Edizioni Cremonense, Rome, 1942,
  pp.~298--302. \MR{0020272}

\bibitem[Mor52]{Mor52}
\bysame, \emph{Sull'unirazionalit\`a dell'ipersuperficie del quarto ordine
  dell'{$S_6$}}, Rend. Sem. Mat. Univ. Padova \textbf{21} (1952), 406--409.
  \MR{66683}

\bibitem[Ott15]{Ott15}
J.~C. Ottem, \emph{Birational geometry of hypersurfaces in products of
  projective spaces}, Math. Z. \textbf{280} (2015), no.~1-2, 135--148.
  \MR{3343900}

\bibitem[Poo17]{Poo17}
B.~Poonen, \emph{Rational points on varieties}, Graduate Studies in
  Mathematics, vol. 186, American Mathematical Society, Providence, RI, 2017.
  \MR{3729254}

\bibitem[Pre49]{Pr49}
A.~Predonzan, \emph{Sull'unirazionalit\`a della variet\`a intersezione completa
  di pi\`u forme}, Rend. Sem. Mat. Univ. Padova \textbf{18} (1949), 163--176.
  \MR{33106}

\bibitem[PS92]{PS92}
K.~Paranjape and V.~Srinivas, \emph{Unirationality of complete intersections},
  Flips and Abundance for Algebraic Threefolds, Ast\'erisque \textbf{211}
  (1992), 241--247.

\bibitem[Ram90]{Ram90}
L.~Ramero, \emph{Effective estimates for unirationality}, Manuscripta Math.
  \textbf{68} (1990), no.~4, 435--445. \MR{1068266}

\bibitem[RS19]{RS19}
F.~Russo and G.~Staglian\`o, \emph{Congruences of 5-secant conics and the
  rationality of some admissible cubic fourfolds}, Duke Math. J. \textbf{168}
  (2019), no.~5, 849--865. \MR{3934590}

\bibitem[Sch18]{Sc18}
S.~Schreieder, \emph{Quadric surface bundles over surfaces and stable
  rationality}, Algebra Number Theory \textbf{12} (2018), no.~2, 479--490.
  \MR{3803711}

\bibitem[Sch19a]{Sc19a}
\bysame, \emph{On the rationality problem for quadric bundles}, Duke Math. J.
  \textbf{168} (2019), no.~2, 187--223. \MR{3909896}

\bibitem[Sch19b]{Sc19b}
\bysame, \emph{Stably irrational hypersurfaces of small slopes}, J. Amer. Math.
  Soc. \textbf{32} (2019), no.~4, 1171--1199. \MR{4013741}

\bibitem[Seg60]{Seg60}
Beniamino Segre, \emph{Variazione continua ed omotopia in geometria algebrica},
  Ann. Mat. Pura Appl. (4) \textbf{50} (1960), 149--186. \MR{121698}

\bibitem[Shi95]{Sh95}
I.~Shimada, \emph{A generalization of {M}orin-{P}redonzan's theorem on the
  unirationality of complete intersections}, J. Algebraic Geom. \textbf{4}
  (1995), no.~4, 597--638. \MR{1339841}

\bibitem[Tot16]{To16}
B.~Totaro, \emph{Hypersurfaces that are not stably rational}, J. Amer. Math.
  Soc. \textbf{29} (2016), no.~3, 883--891. \MR{3486175}

\bibitem[Ver08]{Ver08}
A.~Verra, \emph{Rationality and unirationality problems: from {U}go {M}orin to
  today}, Rend. Istit. Mat. Univ. Trieste \textbf{40} (2008), 165--184 (2009).
  \MR{2583456}

\bibitem[Voi15]{Vo15}
C.~Voisin, \emph{Unirational threefolds with no universal codimension {$2$}
  cycle}, Invent. Math. \textbf{201} (2015), no.~1, 207--237. \MR{3359052}

\end{thebibliography}

\end{document}